\documentclass[10pt]{article}
\usepackage{amsmath,amssymb,amsthm,enumerate,hyperref,graphicx,bbm,appendix,booktabs,enumitem,verbatim}
\usepackage{wrapfig,plain}
\def\Ebox#1#2{%
\begin{center}
\parbox{#1\hsize}{\includegraphics[width= #1\hsize]{#2}} 
\end{center}}

\usepackage[dvipsnames,usenames]{color}
\usepackage{cancel}

\title{On the existence of solutions to stochastic \\quasi-variational
	inequality and complementarity problems}

\author{Uma \ Ravat \and  Uday V.\ Shanbhag\footnote{The authors are affiliated with the
		Department of Statistics at the University of Illinois at
			Urbana-Champaign, IL. and the Industrial and Manufacturing
Engineering at the Pennsylvania State University at University Park,
			PA., respectively.  They are reachable at 
ravat1@illinois.com and  udaybag@psu.edu, respectively. This work was based on
research partially supported by the NSF BECS EFRI  1024837, NSF CMMI (CAREER) 1246887,
		 and DOE
	DE-SC0003879. Finally, we would like to thank Profs. Xiaojun Chen,
	Jong-Shi Pang, and Danny Ralph for their
	suggestions, advice, and encouragement in the context of this research.  }}

\date{\today}

\newcommand{\us}[1]{{\color{black}#1}}
\newcommand{\uss}[1]{{\color{black}#1}}
\newcommand{\uma}[1]{{\textbf{\color{black}#1}}}

\newcommand{\uvs}[1]{{\color{black}#1}}

\usepackage{longtable}
\usepackage{color}

\def\texitem#1{\par\smallskip\noindent\hangindent 25pt
               \hbox to 25pt {\hss #1 ~}\ignorespaces}

\def\st{\mbox{subject to}}

\newcommand{\pmat}[1]{\begin{pmatrix} #1 \end{pmatrix}}

\newtheorem{theorem}{Theorem }
\newtheorem{lemma}[theorem]{Lemma}
\newtheorem{definition}{Definition}[section]
\newtheorem{examplee}{Example}
\newtheorem{proposition}[theorem]{Proposition}
\newtheorem{corollary}[theorem]{Corollary}
\newtheorem{assumption}{Assumption}

\def\zbar{\skew{2.8}\bar z}
\def\limk{\lim_{k\to\infty}}
\def\conv{\mathrm {conv}}
\renewcommand{\emph}[1]{\textbf{#1}}
\def\bkE{{\rm I\kern-.17em E}}
\def\bk1{{\rm 1\kern-.17em l}}
\def\bkD{{\rm I\kern-.17em D}}
\def\bkR{{\rm I\kern-.17em R}}
\def\bkP{{\rm I\kern-.17em P}}

\def\bd{{\mathrm{bd}}}
\def\intt{{\mathrm{int}}}

\def\hnat{{\bf H}^{\textrm{nat}}}
\def\xref{x^{\textrm{ref}}}

\newcommand{\Real}{\ensuremath{\mathbb{R}}}

\def\Dscr{{\cal D}}

\def\Hscr{{\cal H}}

\def\Nscr{{\mathcal N}}

\def\co{\textrm{co}}
\def\om{\omega}

\def\be{\begin{enumerate}}
\def\ee{\end{enumerate}}
\def\st{\mbox{subject to}}

\def\Fscr{{\cal F}}
\def\bkE{\mathbb{E}}
		\def\bk1{{\rm 1\kern-.17em l}}
		\def\bkD{\mathbb{D}}
		\def\bkR{\mathbb{R}}
		\def\bkP{\mathbb{P}}

\def\xni{x_{-i}}

\def\K{{\bf K}}

\def\Pscr{{\mathcal P}}
\def\Dscr{{\mathcal D}}
\def\Kscr{{\mathcal K}}

\def\Fscr{{\mathcal F}}

\def\xref{x^{\textrm{ref}}}
\def\xbar{\bar{x}}

\newcommand{\gap}{\vspace{0.1in}}

\begin{document}
\maketitle
\thispagestyle{empty}
\pagestyle{plain}
\begin{abstract}
Variational inequality problems allow for capturing an expansive class
of problems, including convex optimization problems, convex Nash
games  and  economic equilibrium problems, amongst others. Yet in most
practical settings, such problems are complicated by uncertainty,
motivating the examination of  a stochastic generalization of the variational
inequality problem  and its extensions in which the components of the mapping
contain expectations.  When the associated sets are unbounded, ascertaining
existence requires having access to analytical forms of the expectations.
Naturally, in practical settings, such expressions are often difficult to
derive, severely limiting the applicability of such an approach.  Consequently,
our goal lies in developing techniques that obviate  the need for integration
and our emphasis lies in  developing tractable and verifiable sufficiency
conditions  for claiming existence.  We begin by recapping  almost-sure
sufficiency conditions for stochastic variational inequality problems with
single-valued maps provided in our prior
work~\cite{ravat09characterization,ravat11characterization2} and provide
extensions to multi-valued mappings. Next, we extend these statements to
quasi-variational regimes where maps can be either single or set-valued.
Finally, we refine the obtained results to accommodate stochastic
complementarity problems where the maps are either general or co-coercive.  The
applicability of our results is demonstrated on practically occuring instances
of stochastic quasi-variational inequality problems and stochastic
complementarity problems, arising as nonsmooth generalized Nash-Cournot games
and power markets, respectively.
\end{abstract}


\section{Introduction}\label{sect:intro}

\paragraph{Motivation:} In deterministic regimes, a wealth of
conditions exist for characterizing the solution sets of variational inequality,
quasi-variational inequality and complementarity  problems
(cf.~\cite{cottle92linear,facchinei02finite,konnov07equilibrium}),
including sufficiency statements of existence and uniqueness of solutions as well
as more refined conditions regarding the compactness and connectedness of
solution sets and a breadth of sensitivity and stability questions.
Importantly, the {analytical verifiability of such conditions from
problem primitives (such as the underlying map and the set) is
	essential to ensure the applicability of such schemes, as evidenced
	by the use of such
conditions in analyzing a range of problems arising in power
markets~\cite{hobbs01linear,hobbs07nash,shanbhag11complementarity},
communication
	networks~\cite{pang09nash,pang13sensing,scutari10monotone,yin09nash2},
structural analysis~\cite{pang96complementarity,pang00stability},
amongst others. The first instance of a stochastic
variational inequality problem was presented by King and
Rockafellar~\cite{KingRock93} in 1993 and the resulting  stochastic
variational inequality problem requires an $x \in X$
such that
$$ (y-x)^T \mathbb{E}[G(x,\xi(\omega))] \, \geq \, 0, \quad \forall y \,
	\in \, X, $$
	where $X \subseteq \Real^n$, $\xi:\Omega \to \Real^d$, $G: X \times
	\Real^d \to \Real^n$, $\mathbb{E}[.]$ denotes the expectation
	in a component-wise sense, and
	$(\Omega, {\cal F}, \mathbb{P})$ represents the probability space.  In the decade that followed, there was relatively little
effort on addressing analytical and computational challenges arising
from such problems. But in the last ten years,  there has been an immense interest in the
solution of such {\em stochastic} variational inequality problems via
Monte-Carlo sampling
methods~\cite{jiang08stochastic,juditsky2011,koshal13regularized,shap03sampling}.} 
{But  {\bf verifiable} conditions for characterization of solution
sets have proved to be relatively elusive given the presence of an
	integral (arising from the expectation) in the map.}  Despite the
	rich literature in deterministic settings, direct application of
	deterministic results to stochastic regimes is not straightforward
	and is complicated by several challenges: First,  a direct
	application of such techniques on stochastic problems requires  the
	availability of closed-form expressions of the expectations.
	Analytical expressions for expectation are not {easy} to derive
{even} for single-valued problems with relatively simple {continuous}
distributions. Second, any statement is closely tied to the
distribution. Together, these barriers  severely limit the
generalizability of such an approach. 
To illustrate the
complexity of the problem class under consideration, we consider a
simple stochastic  linear complementarity problem.\footnote{Formal
	definitions for these problems are provided in Section~\ref{sec:21}}

\begin{examplee} \label{ex-svi}
Consider the following stochastic linear complementarity problem:
$$ 0 \leq x \perp \mathbb{E}\left[\pmat{2 & 1 \\ 1 & 2} x + \pmat{-2 + \omega_1 \\ -4 +
\omega_2} \right] \geq 0. $$
Specifically, this can be cast
		  as an affine stochastic variational
			 inequality problem VI$(K,F)$ where
$$ K \triangleq \Real_2^+ \mbox{ and } F(x) \triangleq \mathbb{E}
		  [G(x;\omega)], \mbox{ where } G(x;\omega) \triangleq \left[\pmat{2 & 1 \\ 1 & 2} x + \pmat{-2 + \omega_1 \\ -4 +
\omega_2} \right]. $$
\end{examplee}
{Consider two cases that pertain to either when the  expectation is available in
closed-form (a); or not (b):}
\\
{\bf (a) Expectation $\mathbb{E}[.]$ available in closed form:}  Suppose in this instance, $\omega$ is a random variable that takes
values {$\omega^1$ of $\omega^2$, given by the following}:
$$ \omega^1 = \pmat{1 \\ 1} \mbox{ or } \omega^2 = \pmat{-1 \\ -1} \mbox{ with
	probability } 0.5.$$
Consequently, the stochastic variational inequality problem can be
expressed as 
$$ 0 \leq x \perp \left[\pmat{2 & 1 \\ 1 & 2} x + \pmat{-2 \\ -4} \right] \geq 0. $$
In fact, this problem is a {strongly} monotone linear complementarity
problem and admits a unique solution given by $x^* = (0,2).$ If one
cannot ascertain monotonicity, a common approach lies in examining
coercivity properties of the map; specifically, VI$(K,F)$ is solvable
since there exists an $x^{\rm
	ref} \in K$ such that \uss{(cf.~\cite[Ch.~2]{facchinei02finite})}
	$$ \liminf_{ \|x \| \to \infty, x \in K} F(x)^T(x-x^{\rm ref}) > 0.
	$$
{\bf (b) Expectation $\mathbb{E}[.]$ unavailable in closed-form:}
However, in many practical settings, closed-form expressions of the
expectation are unavailable. \uss{Two possible avenues are available:
\begin{enumerate}
\item[(i)] {If $K$ is compact, under continuity of the expected value
	map, VI$(K,F)$ is solvable.}
\item[(ii)]  {If there exists a single $x \in K$ that solves
	VI$(K,F(.;\omega))$ for almost every $\omega \in \Omega$, VI$(K,F)$
		is solvable. }
\end{enumerate}}
Clearly, if $K$ is a cone and (i) does not hold. Furthermore (ii) appears to be possible only for
		pathological examples and in this case, there does  not exist a single
$x$ that solves the scenario-based VI$(K,{G}(.;\omega))$ for every $\omega
\in \Omega.$ Specifically, the unique solutions to VI$(K,G(.;\omega^1))$ and
VI$(K,{G}(.;\omega^2))$ are $x(\omega^1) = (0,3/2)$
\mbox{ and }  $x(\omega^2) = (1/3,7/3),$ respectively and since
$x(\omega^1) \neq x(\omega^2)$, avenue  (ii) cannot be traversed.
Consequently, neither of the obvious approaches can be adopted yet
VI$(K,F)$ is indeed solvable with a solution given by $x^* = (0,2)$. \\

{Consequently, unless the set is compact or the scenario-based VI is
	solvable by the {\bf same vector} in an almost sure sense}, ascertaining \us{solvability} of stochastic variational inequality
problems for which the {expectation is unavailable in closed form}
does not appear to be immediately possible through known techniques. 
In what could be seen as the first step in this direction, our prior
	work~\cite{ravat11characterization2} examined the solvability of
	convex stochastic Nash games by analyzing the equilibrium
		conditions, {compactly stated as a stochastic variational inequality problem}.
		Specifically, this work relies on {utilizing}  
Lebesgue convergence theorems to develop integration-free sufficiency
conditions that could effectively address stochastic variational
inequality problems, with both single-valued and a subclass of
multi-valued maps arising from nonsmooth Nash games. As a simple illustration of the avenue adopted, consider
Example~\ref{svi} again and assume that {the expectation is
	unavailable in closed form}, we
examine whether the a.s. coercivity property holds (as presented in the
		next section).   It can be
easily seen that there exists an $x^{\rm ref} \in K$, namely $x^{\rm ref}
\triangleq {\bf 0}$, such that 
$$ \liminf_{ \|x \| \to \infty, x \in K} G(x;\omega)^T(x-x^{\rm ref}) >
0, \mbox{ for } {\omega = \omega^1 \mbox{ and } \omega = \omega^2}. 
	$$
{It will be shown that  satisfaction of this coercivity property in an almost-sure
sense is sufficient for solvability.} But such statements, as is natural
with any first step, do not accommodate stochastic quasi-variational
problems and can be refined significantly to accommodate complementarity
problems. Moreover, they cannot accommodate multi-valued variational
maps. {The present work focuses on extending such sufficiency
statements to quasi-variational inequality problems and complementarity
	problems and accommodate settings where the maps are multi-valued.}

\paragraph{Contributions:} {This paper provides amongst the first
attempts to examine { and characterize solutions for } the class of
stochastic quasi-variational inequality and complementarity problems  {
when expectations are unavailable in closed form}.  Our
contributions can briefly be summarized as  follows: \vspace{-0.1in}
\paragraph{{(i) Stochastic quasi-variational inequality problems (SQVIs)}:
} {We begin by recapping our past integration-free statements for stochastic
VIs that required the use of Lebesgue convergence theorems and
variational analysis. Additionally, we provide extensions to  regimes with
multi-valued maps and specialize the conditions for settings with
monotone maps and Cartesian sets}. 
We then extend these conditions to stochastic quasi-variational
inequality  problems where  in addition to a coercivity-like
property, the set-valued mapping needs to satisfy continuity, apart from
other ``well-behavedness'' properties to allow for concluding
solvability. Finally, we extend the sufficiency conditions to
accommodate multi-valued maps.  \vspace{-0.1in}}
\paragraph{{(ii) Stochastic complementarity problems (SCPs):}}
{Solvability of complementarity problems over cones requires a
significantly different tack. We show that analogous verifiable
integration-free statements can be provided for stochastic
complementarity problems. Refinements of such statements are also
provided in the context of co-coercive maps.}  \vspace{-0.1in}

\paragraph{(iii) Applications:} \uss{Naturally, the utility of any
sufficiency conditions is based on its level of applicability. We
describe two application problems. Of these, the first is a  nonsmooth stochastic Nash-Cournot
game which leads to an SQVI while the second is a stochastic equilibrium
problem in power markets which can be recast as a stochastic
complementarity problem. Importantly, both application settings  
are modeled with a relatively high level of fidelity. }\\

\noindent {\bf Remark:} \us{Finally, we emphasize three points  regarding the
relevance and utility of the provided statements: (i) First, such techniques
are of relevance when integration cannot be carried out easily and have
{less} utility when sample spaces are finite; (ii) There are
settings where alternate models for incorporating uncertainty have been
developed~\cite{chen12stochastic,luo3,luo2}. Such models assume
relevance when the interest lies in {\em robust} solutions. Naturally,
an expected-value formulation  has less merit in such settings and
	correspondingly, such robust approaches cannot capture risk-neutral
	decision-making.\footnote{A similarly loose dichotomy exists between
		stochastic programming and robust optimization.} Consequently,
the challenge of analyzing existence of this problem cannot be done away
with by merely changing the formulation, since an alternate formulation
may be inappropriate. (iii) Third, we present sufficiency conditions for
solvability. Still, there are simple examples which can be constructed
in finite (and more general) sample spaces where such conditions will
not hold and yet solvability does hold. We believe that this does not
diminish the importance of our results; in fact, this is not unlike
other sufficiency conditions for variational inequality problems. For
instance, we may construct examples where the coercivity of a map may
not hold over the given set~\cite{facchinei02finite} but the variational
inequality problem may be solvable.  However, in our estimation, in
some of the more practically occurring engineering-economic systems,
such conditions do appear to hold,
reinforcing the relevance of such techniques. In
particular, we show such conditions find applicability in a class of risk-neutral and risk-averse stochastic Nash games
in~\cite{ravat09characterization,ravat11characterization2}. In the present work, we show that such conditions can be
employed for analyzing a class of stochastic generalized Nash-Cournot
games with nonsmooth price functions as well as for a relatively more
intricate networked power market in uncertain settings.} Before proceeding to our results we provide a brief history of deterministic and stochastic  variational inequalities.

\paragraph{Background and literature review:} The variational inequality problem {provides} a broad and unifying
framework for the study of a range of mathematical problems including
convex optimization problems, Nash games, fixed point problems, economic
equilibrium problems and traffic equilibrium
problems~\cite{facchinei02finite}. More generally, the concept of an
equilibrium is central to the study of economic phenomena and engineered
systems, prompting the use of the variational inequality
problem~\cite{hartman66some}. Harker and Pang
\cite{harker_finite-dimensional_1990} provide an excellent survey of the
rich mathematical theory, solution algorithms, and important
applications in engineering and economics while a more comprehensive
review of the analytical and algorithmic tools is provided in the recent
two volume monograph by Facchinei and Pang~\cite{facchinei02finite}.

\gap

Increasingly, the deterministic model proves inadequate when contending
with models complicated by risk and uncertainty.   \us{Uncertainty in variational inequality problems has been considered in a breadth of
application regimes, ranging from traffic equilibrium
problems~\cite{chen12stochastic,gwinner12equilibrium}, cognitive radio
networks~\cite{koshal11single,shanbhag13tutorial}, Nash
games~\cite{ravat11characterization2,luo2}, amongst others.} Compared to the field of optimization, where stochastic
programming~\cite{birge11introduction,shapiro09lectures} and robust
optimization~\cite{bental09robust} have provided but two avenues for
accommodating uncertainty in static optimization problems, far less is
currently available either from a theoretical or an algorithmic
standpoint in the context of stochastic variational inequality problems. Much of the efforts in this regime have been
	largely restricted to Monte-Carlo sampling
	schemes~\cite{gurkan99sample, jiang08stochastic,KingRock93,
		koshal10regularized,yousefian11regularized,koshal13regularized,yousefian13regularized,ralph11asymptotic}, and
		a recent broader survey paper on stochastic variational
		inequality problems and stochastic complementarity problems
		\cite{LFukushima10}.
		
\subsection{Formulations}\label{sec:21}

To help explain the two main formulations  for stochastic variational
inequality problems found in literature - the expected value formulation
and {expected residual minimization formulation}; we now {define} the  
variational inequality problem and its generalizations as well as its stochastic
counterparts. Given a set $K \subseteq\Real^n$ and a mapping $F: \Real^n \to \Real^n$, the {deterministic} variational inequality
problem, denoted by VI$(K,F)$,  requires an $x \in
K$ such that 
\begin{align}
\tag{VI$(K,F)$} (y-x)^T  \ F(x) \ \geq \  0, \qquad \forall y \in K. 
\end{align}
The quasi-variational generalization of VI$(K,F)$ {referred to as a
quasi-variational inequality, } emerges when $K$ {is generalized from a
	constant { map} to a set-valued map }$K:\Real^n \to \Real^n$  with closed
	and convex images.  More formally, QVI$(K,F)$, requires an $x \in
	K(x)$ such that 
\begin{align}
\tag{QVI$(K,F)$} (y-x)^T  \ F(x) \ \geq \  0, \qquad \forall y \in K(x). 
\end{align}
If $K$ is a cone, then the variational inequality problem {reduces to} a
complementarity problem, denoted by CP$(K,F)$, a problem that requires an $x \in K$
such that 
\begin{align}
\tag{CP$(K,F)$} K \ \ni \ x \ \perp \  F(x)  \ \in \ K^*
\end{align}
where $K^* \triangleq \{y: y^Td \geq 0, \, \forall d \in K\}$ and $y \perp w$
implies $y_i w_i = 0$ for $ i = 1, \hdots, n.$ In settings complicated by uncertainty, stochastic
generalizations of such problems are of particular relevance. Given a
continuous probability space $(\Omega,\Fscr, \mathbb{P})$, let
$\mathbb{E}[\bullet]$ denote the expectation operator with respect to
the measure $\mathbb{P}$.  Throughout this paper, we often denote the expectation $\bkE
\left[{G}(x;\xi(\om))\right]$ by $F(x)$, where ${G}(x;\xi(\om))$ denotes the
scenario-based map, $F_i(x) \triangleq
\mathbb{E}[{G}_i(x;\xi(\omega))]$ and $\xi:\Omega \to \Real^d$ is a
$d-$dimensional random variable. For notational ease, we refer to
${G}(x;\xi(\om))$ as $G(x;\om)$ through the entirety of this paper.

\subsubsection{The expected-value formulation}\label{SVI-exp}	
We may now formally  define the {\em stochastic variational inequality problem} as an expected-value formulation. We consider this formulation in the analysis in this paper. 

\begin{definition}[{\bf Stochastic variational inequality problem
	(SVI$(K,F)$)}] \em Let $K \subseteq \Real^n$ be a closed and convex
	set,  {$G: \Real^n \times \Omega \to \Real^n$} be a single-valued map
	and $F(x) \triangleq \mathbb{E}[G(x;\omega)].$ Then the {stochastic variational inequality
		problem}, denoted by \mbox{SVI}($K,F$), requires an $x \in K$
			such that the following holds:
	\begin{align}
(y-x)^T \ F(x) & \ \geq
	\ 0, \quad \forall y \in K, \qquad \mbox{ where } F(x) \triangleq
	\mathbb{E}[G(x;\omega)]. \tag{SVI$(K,F)$}
\end{align}\hfill \qed
\end{definition}
Figure~\ref{svi} provides a schematic of the stochastic variational
inequality problem. Note that when $x$ solves SVI$(K,F)$, there may exist $\omega
\in \Omega$ for which there exist $y \in K$ such that
$(y-x)^TG(x;\omega) < 0$. 
\begin{wrapfigure}{r}{2.8in}
	\vspace{-0.4in}
	\Ebox{1}{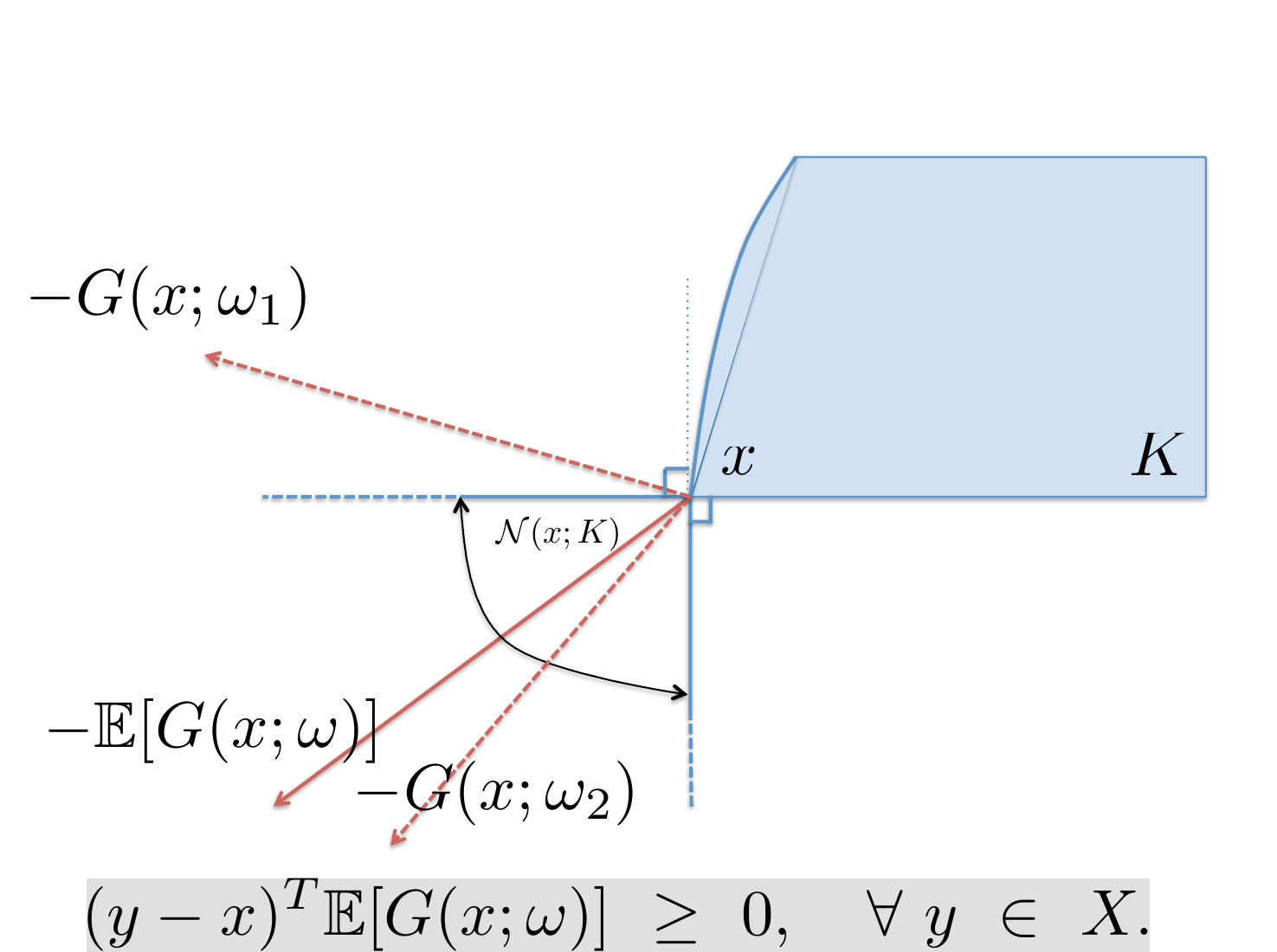}
 	\vspace{-0.1in}
\caption{The stochastic variational inequality (SVI$(K,F)$): The
	expected-value formulation}
	\label{svi}
	\vspace{-0.3in}
	\end{wrapfigure}
Naturally, in instances where the expectation is simple to evaluate, as seen in Example ~\ref{svi} earlier, the resulting SVI$(K,F)$
is no harder than its deterministic counterpart. \us{For instance, if the sample space $\Omega$ is
\us{finite}, then the expectation reduces to a finite summation of
deterministic maps which is itself a deterministic map.  Consequently,
			  the {\bf analysis} of this problem is as challenging as a
deterministic variational inequality problem. } Unfortunately, 
in most stochastic regimes, this evaluation relies on a
multi-dimensional integration and is not a straightforward task. In
fact, a more general risk functional can be introduced instead of the
expectation leading to a risk-based  variational inequality problem { that} 
requires an $x$ such that
\begin{equation*} 
( \, y - x \, )^T \, \rho[G(x;\omega)] \, \geq \, 0, \quad \forall \, y \, \in \, K,
\end{equation*}
{where $\rho[G(x;\omega)] \triangleq \mathbb{E} [{G}(x;\omega)] + \mathbb{D} [
	G(x;\omega)]$ and $\mathbb{D}[\bullet]$ is a map incorporating dispersion measures
	such as standard deviation, upper semi-deviation, or the conditional
	 value at risk (CVaR) (cf.\
			 \cite{RUZabarankin06a,RUZabarankin06b,shapiro09lectures}
for recent advances in the optimization of these risk measures).}
	 \gap

Extensions to set-valued and {specializations to} conic regimes follow naturally. For
instance, if $K$ is a point-to-set mapping defined as $K:\Real^n \to
\Real^n$, then the resulting problem is a stochastic quasi-variational
inequality, and is denoted by SQVI$(K,F)$. When $K$ is  a cone, then
VI$(K,F)$ is equivalent to a complementarity problem CP$(K,F)$ and its
stochastic generalization is given next. 
\begin{definition}[{\bf Stochastic complementarity problem
	(SCP$(K,F)$)}]  \em Let $K$ be a closed
and convex cone in $\Real^n$, {$G:\Real^n \times \Omega \to \Real^n$ }  be
a single-valued mapping and $F(x) \triangleq \mathbb{E}[G(x;\omega)].$ Then the {stochastic complementarity
		problem}, denoted by SCP($K,F)$, requires an $x
		\in K$ such that
\begin{align*}
K \ \ni \ x \ \perp \ F(x)  \ \in \ K^*,  \mbox{ where } F(x) = \bkE\left[G(x;\om)\right].
\end{align*}\hfill \qed
\end{definition}

If the integrands of the expectation {($G(x;\om)$)} are multi-valued
instead of single-valued, then we denote the mapping by $\Phi$.
Accordingly, the associated variational problems are denoted by SVI$(K,
\Phi)$, SQVI$(K,\Phi)$ and {SCP$(K,\Phi)$}. 

	\gap

	The origin of the expected-value formulation can be traced to a paper by King and Rockafellar
	\cite{KingRock93} where  the authors considered  
a {\em generalized equation} \cite{Robinson79,Robinson80}
with {an expectation-valued mapping}. Notably, the analysis and
	computation of the associated solutions are hindered significantly
	when the expectation is over a general measure space. Evaluating
	this integral is challenging, at best, and it is essential that
	specialized analytical and computational techniques be developed
	for this class of variational problems.  From an analytical standpoint,
	 Ravat and Shanbhag have developed existence statements for equilibria of stochastic Nash games that obviate
	 the need to evaluate the expectation by
	combining Lebesgue convergence theorems with standard existence
	statements ~\cite{ravat09characterization,ravat11characterization2}.  {Our earlier worked focused  on Nash games and examined such settings with
nonsmooth payoff functions and stochastic constraints. It
represents a starting point for our current study where we focus on more
general variational inequality and complementarity  problems and their generalizations and
refinements. Accordingly, this paper is motivated by the need to provide
sufficiency conditions for stochastic variational
inequality problems, stochastic quasi-variational inequality problems,
		   and stochastic complementarity problems. In addition, we
		   consider variants when the variational maps are complicated by
multi-valuednes.}  Stability statements
		for stochastic generalized equations have been provided by Liu,
		R\"{o}misch, and Xu~\cite{liu13quantitative}.\\


	 From a computational standpoint, sampling approaches have addressed
	 analogous stochastic optimization problems
	 effectively~\cite{kushner03stochastic,shap03sampling}, but have
	 focused relatively less on variational problems.  In the latter context, there have been two distinct
	threads of research effort. Of these, the first employs
	sample-average approximation schemes~\cite{shap03sampling}. In such
	an approach, the expectation is replaced by a sample-mean and the
	effort is on the asymptotic and rate analysis of the resulting
	estimators, which are obtainable by solving a deterministic
	variational inequality problem
	(cf.~\cite{gurkan99sample,ralph11asymptotic,xu10sample,xu13stochastic}).
	A rather different track is adopted by Jiang and
	Xu~\cite{jiang08stochastic} where a stochastic approximation scheme
	is developed for solving such stochastic variational inequality
	problems. Two regularized counterparts were presented by Koshal,
	Nedi\'{c}, and Shanbhag~\cite{koshal10regularized,koshal13regularized} where two
	distinct regularization schemes were overlaid on a standard
	stochastic approximation scheme (a Tikhonov regularization and a
			proximal-point scheme), both of which allow for almost-sure
	convergence. Importantly, this work also presents distributed
	schemes that can be implemented in networked regimes. A key
	shortcomings of standard stochastic approximation schemes is the
	relatively ad-hoc nature of the choice of steplength sequences.
	In~\cite{yousefian13acc}, Yousefian, Nedi\'c,
	and Shanbhag develop distributed stochastic approximation schemes
	where users can independently choose a steplength rule. Importantly,
	these rules collectively allow for minimizing a suitably defined
	error bound and are equipped with almost-sure convergence
	guarantees. \us{Finally, Wang et al.~\cite{lin1} focus on developing a sample-average approximation method for expected-value formulations of
	the stochastic variational
inequality problems while Lu and Budhiraja~\cite{lu13confidence} examine the confidence statements associated
with estimators associated with a sample-average approximation scheme
for stochastic variational inequality problem, again with
expectation-based maps.}

\subsubsection{The {expected-residual minimization (ERM)} formulation}\label{SVI-stat}

{While there are many problem settings where expected-value formulations
are appropriate (such as when modeling risk-neutral decision-making in a
		competitive regime) but there are also instances where the expected-value
formulation proves {inappropriate}. A case in point arises when attempting to
obtain solutions to a variational inequality
problem that are robust to parametric uncertainty; such problems might
arise when faced with traffic equilibrium or structural design problems.
Given a random mapping ${G}$,
the almost-sure formulation of the stochastic variational inequality
problem requires a (deterministic) vector $x \in K$ such that for almost every $\omega
\in \Omega$, 
\begin{equation} \label{eq:asvi}
(y-x)^T G(x;\omega)  \, \geq \, 0, \quad \forall \ y \, \in \, K.
\end{equation}
Naturally, if 
 $K$ is an $n-$dimensional cone, then \eqref{eq:asvi} reduces to
 CP$(K,{G(.;\omega)})$, a problem that requires an $x$ such that for almost all $\omega \in \Omega$, 
\begin{equation} \label{eq:ascp}
K \, \ni \, x \, \perp \, G(x;\omega) \, \in \, K^*.
\end{equation}
Yet, we believe that obtaining such an $x$ is possible only in
pathological settings. Instead, one approach for resolving such a problem has emerged through 
	   the minimization of  the expected residual. 
For instance, if $K \triangleq \Real^n_+$, this problem is a nonlinear complementarity
problem (NCP) and for a fixed but
arbitrary realization $\omega \in \Omega$, the residual of this system 
can be minimized as follows:
\[
\displaystyle{
{\operatornamewithlimits{\mbox{minimize}}_{x \geq 0}}
} \ \| \, \mathbf{\Phi}(x;\omega) \|,
\]
where $\mathbf{\Phi}(x;\omega)$ denotes the equation reformulation of
the NCP (See~\cite{facchinei02finite}).  
If the Fischer-Burmeister $\phi_{\rm FB}$ is chosen as the C-function,
the  {\em expected residual minimization} (ERM) problem in
\cite{CFukushima05,FCFukushima07} solves the following stochastic
program to compute a solution of the stochastic NCP (\ref{eq:ascp}):
\begin{equation}\label{eqn:ERM}
\begin{array}{ll}
\displaystyle{
{\operatornamewithlimits{\mbox{minimize}}_{x \geq 0}}
} & \mathbb{E} \left[\left\| \, \Phi_{\rm FB}(x;\omega)  \, \right\|\right],  \\ [7pt]
\mbox{where} & \Phi_{\rm FB}(x;\omega) \, \triangleq \, \left( \,
\sqrt{x_i^2 + {G}_i(x;\omega)^2} - \left( \, x_i + {G}_i(x;\omega) \, \right) \,
\right)_{i=1}^n;
\end{array} \end{equation}
see \cite[equation~(3.8)]{LFukushima10}. \us{In~\cite{luo1}, Luo and
Lin  minimize the expected residual while convergence
analysis of the expected residual minimization (ERM) technique has been
carried out in the
context of stochastic Nash games~\cite{luo2} and stochastic variational
inequality problems~\cite{luo3}}. In more recent work, Chen,
	Wets, and Zhang~\cite{chen12stochastic} revisit this problem and
	present an alternate ERM formulation, with the intent of developing
	a smoothed sample average approximation scheme. \us{In contrast with the
expected-value formulation and the almost-sure formulation, Gwinner and
Raciti~\cite{gwinner06random,gwinner12equilibrium} consider an
infinite-dimensional formulation of the variational inequality for
capturing randomness and provide discretization-based approximation
procedures for such problems. }

\gap

The remainder of the paper is organized as follows. In
section~\ref{ss:notation}, we outline our assumptions used,  motivate
our study by considering two application instances, and \uss{provide the
relevant background on integrating set-valued maps}. In section
\ref{sect:SVI}, we recap our sufficiency conditions for the solvability of
stochastic variational inequality problems with single and multi-valued
mappings and  we provide results for the
stochastic quasi-variational inequality problems with single and
multi-valued mappings. Refinements of the statements for SVIs are
provided for the complementarity regime in Section \ref{sect:SCP} under
varying assumptions on the map.  Finally, in section
\ref{sect:examples_applied}, we revisit the motivating examples of
section \ref{ss:notation} and verify that the results developed in this
paper are indeed applicable.

\section{Assumptions, examples, and background}\label{ss:notation}
In Section~\ref{sec:22}, we provide a summary of
the main assumptions employed in the paper. The utility of such models
\us{is demonstrated}  by discussing some motivating examples in
Section~\ref{sec:23}. Finally, some background is provided on the
integrals of set-valued maps in Section~\ref{ss:set-valued background}.

\subsection{Assumptions}\label{sec:22}
We now state the main assumptions used throughout the paper and refer to them
when appropriate. The first of these pertains to the probability space.
\begin{assumption}[Nonatomicity of measure $\bkP$]\label{nonatomic}\em
The probability space $\Pscr = ( \Omega, \Fscr, \bkP)$ is nonatomic.
\end{assumption}
The next  assumption focuses on the properties of the single-valued map, referred to as
$F$. 
\begin{assumption}[Continuity and integrability of
$F$]\label{assumption1}\em
\begin{itemize}
\item[]
\item[(i)] $\us{G}: \Real^n \times \Omega \to \Real^n$ is a single-valued
map. Furthermore, $G(x;\om)$  is continuous in $x$ for almost every $\om
\in \Omega$ and is integrable in $\om$, for every $x$.
\item[(ii)]  $F(x) = \mathbb{E}[G(x;\omega)]$ is continuous in $x$.
\end{itemize}
\end{assumption}
Note that, the assumption of Lipschitz continuity of $G(x;\om)$ with an
integrable Lipschitz constant implies that $\bkE[G(x;\om)]$ is also
Lipschitz  continuous. {We now recall the definition of integrably
	bounded set-valued maps to be used in the assumption to follow.
\begin{definition}[Integrably bounded set-valued map]
\em A set-valued map
	 $H$ that maps from $\Omega$ into nonempty, closed subsets of $\Real^n$ is integrably bounded if there exists a nonnegative integrable function $k \in L^1(\Omega, \Real ,\bkP)$ such that $$H(\om) \subseteq k(\om) B(0,1) \mbox{ almost everwhere}.$$
\end{definition}
}
The next two assumptions pertain to the
set-valued maps employed in this paper. 
When the map is multi-valued, to avoid confusion, we employ the notation
$\Phi(x)$, defined as $\Phi(x) \triangleq \mathbb{E}[\Phi(x;\om)]$, and impose the following assumptions. 
\begin{assumption}[Lower semicontinuity and integrability of
$\Phi$]\label{lsc}\em $\Phi: \Real^n \times \Omega \to 2^{\Real^n}$ is a
set-valued map satisfying the following:
\begin{itemize}
\item[(i)]  $\Phi(x;\om)$  has nonempty and closed images for every $x$
and every $\omega \in \Omega$.
\item[(ii)] \hspace{-0.021in}$\Phi(x;\om)$ is lower semicontinuous in $x$ for almost all
$\om \in \Omega$ and integrably bounded for every $x$.
\end{itemize}
\end{assumption}
Finally, when considering quasi-variational inequality problems, $K$ is
a set-valued map, rather than a constant map. 
\begin{assumption}\label{assumption-det-stratset}\em
The set-valued map $K: \Real^n \to 2^{\Real^n}$ is deterministic, closed-valued, convex-valued. 
\end{assumption}

We conclude this subsection with some notation. $\textrm{cl}(U)$ denotes the closure of a set $U \subset
\Real^n$, $\textrm{bd}(U)$ denotes the boundary of the set $U$ and
$\textrm{dom}(K)$ denotes the domain of the mapping $K$.

\subsection{Examples}\label{sec:23}

We  now provide two instances of where stochastic
variational  problems  arise in practice. 

\paragraph{Nonsmooth stochastic Nash-Cournot equilibrium problems:}\label{ss:ex:STEP}

Cournot's oligopolistic model is amongst the most widely used
models for modeling strategic interactions between a set of
noncooperative firms~\cite{fudenberg91game}. Under an assumption that firms produce a
homogenous product, each firm attempts to maximize profits by making a
quantity decision while taking as given, the quantity of its rivals.
Under the Cournot assumption, the price of the good is assumed to be
dependent on the aggregate output in the market. The resulting Nash
equilibrium, qualified as the Nash-Cournot equilibrium, represents a set
of quantity decisions at which no firm can increase  profit through a
unilateral change in quantity decisions.  unilaterally changing the
quantity of the product it produces. {To accommodate uncertainty in
costs and prices in Nash-Cournot models and loss of
differentiability of price functions which can occur for example
		by introduction of price caps ~\cite{hobbs07nash} we consider a
		stochastic generalization of the classical deterministic
		Nash-Cournot model and allow for piecewise smooth price
		functions, as captured by the following assumption on costs and
prices.}
\begin{assumption}\label{nc}\em
Suppose the cost function $c_i(x_i)$  is an increasing 
convex {twice}
continuously differentiable function for all  $i = 1, \hdots, N$.
Let $X \triangleq \sum_{i=1}^n x_i$. Since $x_i$ denotes the quantity produced, $x_i \geq 0$. The price function $p(X;\omega)$ is a decreasing piecewise smooth convex function
	where $p(X;\om)$ is given by 
	\begin{align}\label{eqn:nc_price}
	p(X;\omega) =  \begin{cases}
									p^1(X;\om), & 0 \leq X \leq \beta^1 \\
									p^j(X;\om),  & \beta^{j-1} \leq X \leq \beta^j, j =
									2, \hdots, s \\
									p^s(X;\om),& \beta^s \leq X
								\end{cases}
	\end{align}
	 where $p^j(X;\om) = a^j(\omega) - b^j(\omega) X$ is a strictly
	 decreasing affine function of $X$ for $j = 1, \hdots, s$. Finally,  
	$\beta^1,\hdots, \beta^s$ are a set of increasing positive scalars
	and $(a^j(\omega),b^j(\om))$ are positive in an almost-sure sense and integrable for $j = 1,
	\hdots, s.$
\end{assumption} 
Consider an $N-$player generalized Nash-Cournot game.  Given the tuple
of rival strategies $x_{-i}$, the $i$th
player's strategy set is given by $K_i(x_{-i})$ while his objective
function is given by
$\bkE [f_i(x;\om)] \triangleq c_i(x_i) - \bkE [p(x;\om)x_i]$.
 Then $ \{x_i^*\}_{i=1}^N$ denotes a stochastic Nash-Cournot equilibrium if $x_i^*$ solves the convex optimization
problem G$_i(x^*_{-i})$, defined as 
\[ \begin{array}{ll} 
\displaystyle {\operatornamewithlimits{\mbox{minimize}}_{x_i}}
 & \displaystyle \mathbb{E}  [f_i(x;\om)] \\
				\st	&	x_i \in K_i(x_{-i}).
\end{array} \]
The equilibrium conditions of this problem are given a stochastic QVI
 with multi-valued mappings. In
section~\ref{sect:examples_applied}, we revisit this problem with the intent of developing {statements about existence of solutions}

\paragraph{Strategic behavior in power markets:}\label{ss:ex:NC}
Consider a power market model in which a collection of generation firms
compete in a single-settlement market. Economic equilibria in power markets has been extensively studied using
a complementarity framework; see \cite{hobbs07nash,
hobbs04complementarity}. Our model below is based on the model of
Hobbs and Pang \cite{hobbs07nash} which we modify to account for
uncertainty in prices and costs. 

 \gap

 Consider a set of nodes $\Nscr$ of a network. The set of generation
 firms is indexed by $f$, where $f$ belongs to the finite set $\Fscr$.
 At a node $i$ in the network, a firm $f$ may generate $g_{fi}$ units at
 node $i$ and sell $s_{fi}$ units to node $i$. The total amount of power
 sold to node $i$ by all generating firms is $S_i$. The generator firms'
 profits are revenue less costs.  If the nodal power price at node $i$th
 is a random function given by $p_i(S_i;\om)$ where $p_i(S_i;\om)$ is a
 decreasing function of $S_i$, then the firms' revenue is just the price
 times sales $s_{fi}$.  The costs incurred by the firm $f$ at node $i$
 are the costs of generating $g_{fi}$ and transmitting the excess
 $(s_{fi} -g_{fi})$.  Let the cost of generation associated with firm
 $f$ at node $i$ be given by $c_{fi}(g_{fi};\omega)$ and the cost of
 transmitting power from an arbitrary node (referred to as the hub) to
 node $i$ be given by $w_i$. The constraint set ensures a balance
 between sales and generation at all nodes, nonnegative sales and
 generation and generation subject to a capacity limit.  The price and
 cost functions are assumed to satisfy the following requirement.
 
 \begin{assumption}\label{assump:power}\em
For $i \in \Nscr$, the price function $p_i(S_i;\om)$ is a decreasing
function { with its  absolute value } bounded above by a nonnegative integrable function 
$\bar p_i(\omega)$. Furthermore, the cost functions $c_{fi}(g_{fi};\om)$ are
nonnegative and  increasing. 
\end{assumption}

\gap

The resulting problem faced by generating firm $f$ requires determining
sales $s_{fi}$ and generation $g_{fi}$ at all nodes $i$ and is captured
as follows:
\[ 
\begin{array}{ll} 
\displaystyle{
{\operatornamewithlimits{\mbox{maximize}}_{s_{fi}, \, g_{fi}}}
} & \mathbb{E}\left[\displaystyle{
\sum_{i\in \Nscr}
} \, \left( \, p_i(S_i;\om)s_{fi} - c_{fi}(g_{fi};\om)
		-(s_{fi}-g_{fi})w_i \, \right)\right] \\ [0.2in]
\mbox{subject to} &
\left\{ \begin{array}{lll}
0 & \leq & g_{fi} \, \leq \, {\rm cap}_{fi} \\ [5pt]
0 & \leq & s_{fi} \end{array} \right\}, \quad \forall \, i \in \Nscr \\ [0.4in]
\mbox{and} & \displaystyle{
\sum_{i \in \Nscr}
} \, ( \, s_{fi} - g_{fi} \, ) \, = \, 0.
\end{array}
\]
Note that, the generating firm sees the transmission fee $w_i$ and the rival firms' sales $s_{-fi} \equiv \{s_{hi}~:~ h \neq f\}$ as exogenous parameters to its optimization problem even though they are endogenous to the overall equilibrium model as we will see shortly.  The ISO sees the transmission fees $w=(w_i)_{i \in \Nscr}$ as exogenous and prescribes flows $y = (y_i)_{i \in \Nscr}$ as per the following linear program
\[ \begin{array}{ll} 
\displaystyle{
{\operatornamewithlimits{\mbox{maximize}}}
} & \displaystyle{
\sum_{i \in \Nscr}
} \, y_i w_i \\ [0.2in]
\mbox{subject to} & \displaystyle{
\sum_{i \in \Nscr}
} \, {\rm PDF}_{ij} y_i \, \leq \, T_j, \qquad \forall \, j \in \Kscr,
\end{array} \]
where $\Kscr$ is the set of all arcs or links in the network with node set $\Nscr$, $T_j$ denotes the transmission capacity of link $j$, $y_i$ represents the transfer of power (in MW) by the system operator from a hub node to node node $i$ and  PDF$_{ij}$
denotes the power transfer distribution factor,
which specifies the MW flow through link $j$ as a consequence of
unit MW injection at an arbitrary hub node and a unit withdrawal at node $i$.

\gap

Finally,  to clear the market, the transmission flows $y_i$ must must balance the net sales at each node:
\[
y_i \, = \, \displaystyle{
\sum_{f \in \Fscr}
} \, \left( \, s_{fi} - g_{fi} \, \right), \qquad \forall \, i \in \Nscr.
\]

The above market equilibrium  problem  which comprises of each firm's problem, the ISO's problem and market-clearing condition, can be expressed as a stochastic complementarity problem by following the technique from \cite{hobbs07nash}. This equivalent formulation of the above market equilibrium problem is described in the last section \ref{sect:examples_applied}. We also illustrate the  solvability of such problems using the framework developed in this paper.

\subsection{Background on integrals of set-valued mappings}\label{ss:set-valued background}

 {Recall that by Assumption~\ref{nonatomic}, $(\Omega,\Fscr, \bkP)$ is a
	 nonatomic continuous probability space.} Consider a set-valued map
	 $H$ that maps from $\Omega$ into nonempty, closed subsets of $\Real^n$ . We
	 recall three definitions from \cite[Ch.~8]{aubin90setvalued}
	 \begin{definition}[Measurable set-valued map]
\em	 A map $H$ is measurable if the inverse image of each open set in $\Real^n$ is
a measurable set: for all open sets $O\subseteq \Real^n $, we have
$$ H^{-1}(O) = \left\{ \om \in \Omega \mid H(\om) \cap O \neq \emptyset \right\} \in \Fscr.$$
\end{definition}


\begin{definition}[Measurable
selection]
\em Suppose a measurable map $h:\Omega \to \Real^n$ satisfies $h(\om) \in
 H(\om)$  for almost all $ \om
 \in \Omega$. Then $h$ is called a measurable selection of
 $H$.
\end{definition}

The existence of a measurable selection is proved in \cite[Th.~8.1.3]{aubin90setvalued}.
\begin{definition}[Integrable selection]\em
 A measurable selection $h:\Omega \to \Real^n$ is an integrable
 selection if $\mathbb{E} [h(\omega)] < \infty$ where
 $$ \mathbb{E}[h(\omega)] = \int_\Omega h d\bkP < \infty.$$
\end{definition}
The set of all integrable selections of $H$ is denoted by $\Hscr$ and is
	defined as follows:
\begin{align*}
	\Hscr & \triangleq \left\{ h \in L^1(\Omega, \Fscr,\bkP):  h(\om)\in
	H(\om) \hbox{ for almost all } \om \in \Omega\right\}, 
	\end{align*}
\begin{definition}[Expectation of a set-valued map]
\em The expectation of the set-valued map $H$, denoted by
$\bkE[H(\om)]$, is the set of integrals of integrable selections of $H$:
$$\bkE[H(\om)] \triangleq \int_\Omega H d\bkP \triangleq  \left\{ \int_\Omega h
d\bkP \mid h \in \Hscr \right\}.$$
\end{definition}
{If the images of $H(\omega)$ are convex then this set-valued integral
is convex~\cite[Definition 8.6.1]{aubin90setvalued}.  If the assumption
of convexity of images of $H$ does not hold, then the convexity of this
integral follows from Th.~8.6.3~\cite{aubin90setvalued}, provided that
the probability measure is nonatomic.} We make precisely such an
assumption (See Assumption~\ref{nonatomic}) and are therefore
guaranteed that the integral of the set-valued map $H$ is a
convex set \cite[Th.~8.6.3]{aubin90setvalued}. 

\gap

Recall that a point $\bar{z}$ of a convex set $K$ is said to be
extremal if there are no two points $x,y \in K$ such that $\lambda x +
(1-\lambda) y = \bar{z}$ for $\lambda \in (0,1)$ and is denoted by
$\zbar \in \textrm{ext} (K).$ Similarly, as per
Def.~8.6.5~\cite{aubin90setvalued}, we define an extremal selection as
follows:
\begin{definition}[Extremal selection]\em
Given the convex set $\int_\Omega H d\bkP$, an integrable selection $h \in \Hscr$ is an extremal
selection of $H$ if $$ \int_{\Omega} h d\bkP \hbox{ is an extremal
	point of the closure of the convex set } \int_{\Omega} H d\bkP. $$
\end{definition}
The set of all extremal selections is denoted by $\Hscr_e$ and is
defined as follows: 
$$ \Hscr_e \triangleq \left\{ h \in \Hscr\mid \int_{\Omega} h d\bkP
\in \textrm{ext}  \left( \textrm{cl} \left(\int_{\Omega} H
			d\bkP \right) \right)\right\}. $$
By Theorem ~\cite[Th.~8.6.3]{aubin90setvalued}, we have the following
Lemma for the representation of extremal points of closure of
$\int_\Omega H d \bkP$.
\begin{theorem}[Representation of extreme points of set-valued integral
{\cite[Th.~8.6.3]{aubin90setvalued}}]\label{th:rep_ext_pt} \em Suppose
Assumption~\ref{nonatomic} holds and let $H$ be a measurable set-valued
map from $\Omega$ to subsets of $\Real^n$  with nonempty closed images.
Then the following hold: 
\begin{itemize}
\item[(a)] $\int_\Omega H d\bkP $ is convex and extremal points of
${\textrm{cl}}(\int_\Omega H d\bkP)$ are contained in $\int_\Omega H d\bkP$.
\item[(b)] If $ x \in {\textrm{ext}}\left( {\textrm{cl}}\left(\int_\Omega H  d\bkP\right)\right) $, then there exists a unique $h \in \Hscr_e$ with $x = \int_\Omega h d\bkP$.
\item[(c)] If $H$ is integrably bounded, then the integral $\int_\Omega H d\bkP $  is also compact.
\end{itemize}
\label{th:int_cpt}
\end{theorem}

As a corollary to the above theorem, we have a representation of points
in a set-valued integral as follows.
\begin{corollary}[Representation of points in a set-valued integral {\cite[Th.~8.6.6]{aubin90setvalued}}]\label{th:rep_ele_integral} \em
Let $H$ be a measurable integrably bounded set-valued map from $\Omega$
to subsets of $\Real^n$ with nonempty closed images. If $\bkP$ is
nonatomic, then for every $x \in \int_\Omega H d\bkP$, there exist $n+1$ extremal selections $h_k \in \Hscr_e$ and $n+1$ measurable sets $A_k \in \Fscr$, $k =0,\cdots, n$, such that
$$x = \displaystyle \int_\Omega \left( \sum_{k=1}^n \chi_{\tiny A_k}h_k \right) d\bkP $$
where $\chi_{A_k}$ is the characteristic function of $A_k$.
\end{corollary}

\section{Stochastic quasi-variational  inequality problems}\label{sect:SVI}

In this section, we develop sufficiency conditions for the solvability
of stochastic quasi-variational  inequality problems under a diversity of
assumptions on the map. More specifically, we begin by recapping
sufficiency conditions for the solvability of stochastic variational inequality problems with
single-valued and multi-valued maps in Section~\ref{subsect:SVIF}. In many settings, the variational inequality problems may prove
incapable of capturing the problem in question. For instance, the
equilibrium conditions  of convex  generalized Nash games are given by a
quasi-variational inequality problem. As mentioned earlier, when the
constant map $K$ is replaced by a set-valued map $K: \Real^n \to
\Real^n$, the resulting problem is an SQVI. In this section, we extend
the sufficiency conditions presented in the earlier section to
accommodate the SQVI$(K,F)$ (Section \ref{subsect:sqvi-single})
and SQVI$(K,\Phi)$ (Section \ref{subsect:sqvi-multi}), respectively.
Throughout this section, Assumption \ref{assumption-det-stratset} holds
for the set-valued map $K$.
\subsection{SVIs with single-valued and multi-valued mappings}\label{subsect:SVIF}
In this section, we \uss{begin by  assuming that} the scenario-based mappings {$G(\bullet;\om)$} are
single-valued for each $\om \in \Omega$. With this assumption, we
provide sufficient conditions that the scenario-based {VI$(K,G(\bullet; \om))$}
must satisfy in order to conclude the  existence of solution to the
stochastic SVI($K,F$) without  requiring  the evaluation of expectation
operator. {Recall that in SVI($K,F$), $F(x) = \bkE[G(x;\om)]$.} In
particular, in the next proposition, we show that if a certain
coercivity condition holds for the scenario-based map {$G(\bullet;\om)$} in an
almost-sure sense then existence of the solution to the above SVI may be
concluded without resorting to formal evaluation of the expectation. 

\begin{proposition}[Solvability of SVI($K,F$)]\label{prop-smooth-exist}
\em
Consider a stochastic variational inequality  SVI$(K,F)$.  Suppose
Assumption \ref{assumption1} holds.
Suppose there exists an $\xref \in K$ such that  the following hold:
\begin{itemize}
\item[(i)] $ \displaystyle \liminf_{\|x\| \to \infty, x \in K} \left[ G(x;\om)^T(x-\xref)\right] > 0 \text{ almost surely;}$
\item[(ii)] Suppose there exists a nonnegative integrable function
{$u(\omega)$} such that ${G(x;\om)^T(x-\xref)}  \geq -u(\om)$ holds almost surely for
any $x$.
\end{itemize}
	Then the stochastic  variational inequality  SVI$(K,F)$ has a
	solution.
\end{proposition}
\begin{proof}
See appendix.
\end{proof}

In settings where $K$ is a Cartesian product, defined as
\begin{align}
\label{cart}
K \triangleq \prod_{\nu =1}^N K_{\nu},
\end{align} 
  VI$(K,F)$ is a partitioned
variational inequality probem, as defined in
~\cite[Ch.~3.5]{facchinei02finite}.  Accordingly,
Proposition~\ref{prop-smooth-exist} can be weakened so that even if the
coercivity property holds for just one index $\nu \in \{1, \hdots, N\}$,
the stochastic variational inequality is solvable. 
\begin{proposition}[Solvability of SVI($K,F$) for Cartesian
$K$]\label{prop-smooth-exist-cart} \em
Consider a stochastic variational inequality  SVI$(K,F)$ where $K$ is a
Cartesian product of closed and convex sets as specified in
\eqref{cart}. Suppose that Assumption  \ref{assumption1} and  the
following hold:
\begin{itemize}
\item[(i)]  There exists an
	$\xref \in K$ and an index $\nu \in
	\{1,\hdots,N\}$ such that { for any $x \in K$},
$$\liminf_{\|x_\nu\|\to \infty, x_\nu\in K_\nu} \left[ {G}_\nu(x;\om)^T(x_\nu -
		\xref_\nu) \right]> 0 $$ holds in an almost sure sense; and
\item[(ii)] For the above $\nu$ and for any $x$, 
suppose there exists a nonnegative integrable function {$u( \omega)$}
such that ${G_{\nu}(x;\om)^T(x_{\nu}-\xref_{\nu})}  \geq -u(\om)$ holds almost surely for any $x$.
\end{itemize}
Then SVI$(K,F)$ admits a solution.
\end{proposition}
\begin{proof}
See appendix.
\end{proof}

{It is well known that strong monotonicity of the map $F$ and convexity
	of the set $K$ guarantee existence of solution for the
{deterministic}		VI$(K,F)$~\cite[Theorem~2.3.3]{facchinei02finite}. {In fact, it
		may be recalled that if
		$F(x)$ is a strongly monotone map then for {\bf any} reference
		vector $x^{\rm ref} \in K$, we have that 
\begin{align*} & \quad \liminf_{x \in K, \|x \| \to \infty} { F(x)^T(x-x^{\rm ref})} = 
\liminf_{x \in K, \|x \| \to \infty} { (F(x)-F(x^{\rm ref}) +
	F(x^{\rm ref}))^T(x-x^{\rm ref})} \\
& \geq \liminf_{x \in K, \|x \| \to
		\infty} \left[\eta \|x-x^{\rm ref}\|^2 +{F(x^{\rm
				ref})^T(x-x^{\rm ref})}\right] = + \infty, 
\end{align*}
	where the first term tends to $+\infty$ at a quadratic rate while
		the second term may tend to $\pm \infty$ at a linear rate. In
		effect, we have that when $F(x)$ is a strongly monotone map, the
		coercivity requirement holds immediately and existence follows.
It follows that if ${G}(x;\omega)$ is a strongly monotone map for almost
every $\omega$ with constant $\eta(\omega)$ where $\eta(\omega) \geq
\eta$ a.s., then $F(x)$ is a strongly monotone map and {SVI$(K,F)$} is
solvable with a  unique solution.  We now present a similar result for stochastic variational inequalities SVI$(K,F)$ under
the assumption that the mapping $G(x;\om)$ is a monotone  mapping over
$K$ for almost every $\omega \in \Omega$, a weaker set of sufficient
conditions as compared to Propositon \ref{prop-smooth-exist} guarantees
existence of solution for SVI($K,F$).}
\begin{corollary}[Solvability of SVI($K,F$) under monotonicity]
\label{prop-monotone-exist}\em
Consider SVI$(K,F)$ and suppose that Assumption \ref{assumption1} holds. Further, let $G(x;\om)$ be a monotone mapping on $K$ for almost every
$\omega \in \Omega.$ Suppose there exists an $\xref \in K$ such that
the following hold:
\begin{itemize}
\item[(i)] $ \displaystyle \liminf_{\|x\| \to \infty, x\in K} \left[ G(\xref;\om)^T(x-\xref) \right]> 0
$ holds in almost sure sense; 
\item[(ii)] Suppose there exists a nonnegative integrable function
{$u(\omega)$} such that ${G(\xref;\om)^T(x-\xref)}  \geq -u(\om)$ holds almost surely for any $x$.
\end{itemize}
Then SVI$(K,F)$  is solvable.
\end{corollary}
\begin{proof}
See appendix.
\end{proof}
\uss{Next,} we consider a stochastic variational inequality 
SVI$(K,\Phi)$ where $\Phi(x) \triangleq \bkE [\Phi(x;\om)]$  {and
$\Phi(x;\om)$ is a multi-valued mapping. Before proceeding to
		prove existence {of solutions to SVIs with multi-valued maps},
we restate Corollary
\ref{th:rep_ele_integral} for the case of the set-valued integral
$\Phi(x)=\bkE[\Phi(x;\om)]$ of interest. 
\begin{lemma}[{Representation of elements of set-valued integral as an integral of convex combination of extremal selections}]
\em Suppose Assumption \ref{nonatomic} holds. Let $\Phi$ be a measurable  integrably bounded set-valued map from $\Real^n
\times \Omega$ to subsets of $\Real^n$ with closed nonempty images. Then
any  $w \in \bkE[ \Phi(x;\om)]$ can be expressed as 
\begin{align*}
w= \displaystyle \int_{\Omega}g(x;\om) d\bkP
\end{align*}
where $g(x;\om) = \sum_{k=1}^n \lambda_k(x)f_k(x;\om) $ and  $\lambda_k(x) \geq 0 , \sum_{k=0}^n \lambda_k(x) = 1$ and each $f_k(x;\om)$ is an extremal selection of $\Phi(x;\om)$.
\label{lemma:selections}
\end{lemma}
\begin{proof}
Since  $w \in \bkE[ \Phi(x;\om)]$  and $\bkE[ \Phi(x;\om)]$ is a convex set,
	   thus by Carath{$\acute{\textrm{e}}$}odory's theorem  for convex
	   sets, there exists $\lambda_k(x) \geq 0, w_k \in
{\textrm{ext}}\left({\textrm{cl}}\left(\bkE[ \Phi(x;\om)] \right) \right) $
such that $$ w = \sum_{k=0}^n \lambda_k(x) w_k, \quad \sum_{k=0}^n
\lambda_k(x) = 1$$ Now, since $w_k \in
{\textrm{ext}}\left({\textrm{cl}}\left(\bkE[ \Phi(x;\om)] \right) \right)$,
	by {\cite[Th.~8.6.3]{aubin90setvalued}}, for each index $k$,  there exists
	an extremal selection $f_k(x;\om)$ from $\Phi(x;\om)$ such that
	$\int_\Omega f_k(x;\om) d\bkP= w_k$. Thus, we obtain 
$$w = \sum_{k=0}^n \lambda_k(x) \int_\Omega f_k(x;\om) d\bkP,$$
which can be rewritten as $$w =  \displaystyle \int_{\Omega}g(x;\om)
d\bkP$$ where $g(x;\om)= \sum_{k=0}^n \lambda_k(x) f_k(x;\om)$. The
required representation result follows.
\end{proof}

We begin by providing a coercivity-based sufficiency condition {for
	deterministic multi-valued variational inequalities~\cite{Kien2008MultivaluedVI}.}
\begin{theorem}[Solvability of VI $(K,\Phi)${\cite[Th.~2.1,2.2]{Kien2008MultivaluedVI}}]\em
Suppose $K$ is a closed and convex set in $\Real^n$ and let $\Phi:K
	\rightrightarrows \Real^n$ be a lower semicontinuous multifunction with
	nonempty closed and convex images. Consider the following
	statements:
	\begin{enumerate}
	\item [(a)] Suppose there exists an $\xref
	\in K$ such that $L_<(K,\Phi)$ is bounded (possibly empty) where 
\begin{align}\label{eqn:setL}
 L_< (K,\Phi) \triangleq  \left\{ x \in K: \inf_{y \in \Phi(x)}
(x-\xref)^Ty	< 0 \right\}. 
\end{align}
	\item [(b)] The variational inequality VI$(K,\Phi)$ is solvable
	\end{enumerate}
	Then, (a) implies (b). Furthermore, if $\Phi(x)$ is a
	pseudomonotone mapping over $K$, then (a) is equivalent to (b).  
\label{Th:L<bdd}
\end{theorem}

	Using this condition, we proceed to develop sufficiency conditions
	for the existence of solutions to SVI$(K,\Phi).$ 

\begin{proposition}[Solvability of
SVI$(K,\Phi)$]\label{prop:solv-multi-SVI-convex}\em
Consider SVI$(K,\Phi)$ and suppose assumptions \ref{nonatomic} and
\ref{lsc} hold.  Further, suppose the following hold:
\begin{enumerate}
\item[(i)] Suppose there exists an $\xref \in K$ such that
$$ \liminf_{x \in K, \|x\| \to \infty} \left( {\inf_{w \in
			 \Phi(x;\om)} w^T(x-\xref)}\right) >
	0 \mbox{ almost surely}.$$
\item[(ii)] For the above $\xref$,  suppose  there exists a nonnegative integrable function $u(\om)$ such that  $ g(x;\om)^T(x-\xref) \geq -u(\om)$ holds  almost surely for any integrable selection $g(x;\om)$ of $\Phi(x;\om)$ and for any $x$.
\end{enumerate}
		Then SVI$(K,\Phi)$ is solvable. 
\end{proposition}
\begin{proof}
The proof proceeds in two parts.
\begin{itemize}
\item[(a)] We first show that the following coercivity condition holds
for the expected value map: there exists an $\xref \in K$ such that
\begin{align}\label{eqn:coercivity-multi}
 \liminf_{x \in K, \|x\| \to \infty} \left( {\inf_{w \in
			 \Phi(x)} w^T(x-\xref)}\right) >	0.
\end{align}
\item[(b)] If (a) holds, then we show that for the given $\xref
	\in K$, the set $L_<(K,\Phi)$ is bounded (possibly empty) where 
	$L_<(K,\Phi)$ is defined in \eqref{eqn:setL}.
\end{itemize}
\paragraph{Proof of (a):}
We proceed by a contradiction and assume that
\eqref{eqn:coercivity-multi} does not hold for the expected value map.
Thus, for any $\xref$, there exists a  sequence $x_k \in K$ with $\|x_k
\| \to \infty$ such that 
$$ \liminf_{k \to \infty} \left( {\inf_{w \in
			 \Phi(x_k)} w^T(x_k-\xref)}\right) \leq 0. $$
Since $\Phi(x_k)$ is a closed set, the infimum above is attained at $y_k \in \Phi(x_k)$. Thus, we have
\begin{align}\label{eqn:yk}
 \liminf_{k \to \infty}  \left[ y_k^T(x_k-\xref) \right] \leq 0
\end{align}
Now, $y_k \in \Phi(x_k) = \bkE[\Phi(x_k;\om)]$. By the representation
Lemma (Lemma~ \ref{lemma:selections}),
since $y_k \in \Phi(x_k)$, we have that 
\begin{align*}
y_k= \displaystyle \int_{\Omega}g_k(x_k;\om) d\bkP
\end{align*}
for some $g_k(x_k;\om) = \sum_{l=1}^n \lambda_l(x_k)f_l(x_k;\om) $ where  $\lambda_l(x_k) \geq 0 , \sum_{l=0}^n \lambda_l(x) = 1$ and each $f_l(x_k;\om)$ is an extremal selection of $\Phi(x_k;\om)$.
With this substitution, (\ref{eqn:yk}) becomes
\begin{align*}
 \liminf_{k \to \infty}  \left[ \displaystyle \int_{\Omega}
 g_k(x_k;\om)^T(x_k-\xref)  d\bkP \right]\leq 0.
\end{align*}
By hypothesis (ii), we may use Fatou's Lemma to interchange the order of
integration and limit infimum, as shown next: 
\begin{align*}
\displaystyle \int_{\Omega} \liminf_{k \to \infty}
\left[~g_k(x_k;\om)^T(x_k-\xref)  \right]d\bkP \leq 0.
\end{align*}
Consequently, there is a set of positive measure $U \subseteq \Omega$,
	over which the integrand is nonpositive or
\begin{align*}
 \liminf_{k \to \infty} ~\left[  g_k(x_k;\om)^T(x_k-\xref) \right]  \leq 0, \quad \forall \om \in U.
\end{align*}
Substituting the expression for $g_k$, we obtain the following
inequality.
\begin{align*}
 \liminf_{k \to \infty}   ~\left[ (x_k-\xref)^T\left(\sum_{l=1}^n \lambda_l(x_k)f_l(x_k;\om)\right)  \right]  \leq 0, \quad \forall \om \in U.
\end{align*}
As a result, for at least one index  $l \in \{1,...,n\}$, we have that
\begin{align*}
 \liminf_{k \to \infty}   ~ \left[ \lambda_l(x_k)f_l(x_k;\om)^T(x_k-\xref)  \right] \leq 0, \quad \forall \om \in U.
\end{align*}
Since $0 \leq \lambda_l(x_k) \leq 1$, the following must be true for the above $l$:
\begin{align*}
 \liminf_{k \to \infty}   \left[ f_l(x_k;\om)^T(x_k-\xref)\right]   \leq 0, \quad \forall \om \in U.
\end{align*}
Moreover, $f_l(x_k;\om) \in \Phi(x_k;\om)$ since it is an extremal
selection and  we have that 
\begin{align*}
 \liminf_{k \to \infty}   \left[ \inf_{w \in \Phi(x_k;\om)}  w^T(x_k-\xref)   \right] \leq 0, \quad \forall \om \in U.
\end{align*}
Since, this holds for any $\xref$, it holds for  the vector $\xref$ in the
		hypothesis and for a set of positive measure $U$, we have that 
$$ \liminf_{x \in K, \|x\| \to \infty} \left[ {\inf_{w \in
			 \Phi(x;\om)} w^T(x-\xref)}\right] \leq 	0, \forall
			 \omega \in  U.$$
This contradicts the hypothesis and condition
\eqref{eqn:coercivity-multi} must hold for the expected value map.

\paragraph{Proof of (b)}
Next, we show that when the condition \eqref{eqn:coercivity-multi} holds
for the expected value map, for the given $\xref \in K$, the set
$L_<(K,\Phi)$ is bounded (possibly empty) where $L_<(K,\Phi)$ is defined
in \eqref{eqn:setL}. 
\uvs{We proceed by contradiction and assume that} $L_<(K,\Phi)$ is nonempty and unbounded.
Then, there exists a sequence $\{x_k\} \in L_<(K,\Phi)$ with $\|x_k\| \to \infty$. Since $x_k \in L_<(K,\Phi)$, we have for each $k$, $$\inf_{y \in \Phi(x_k)} \left[(x_k-\xref)^Ty	\right]< 0.$$
This implies that for the sequence $\{x_k\}$, we have that
$$ \liminf_{x_k \in K, \|x_k\| \to \infty} \left[ {\inf_{w \in
			 \Phi(x_k)} w^T(x_k-\xref)}\right]  \leq 	0.$$
But this contradicts the coercivity property of the expected value map
proved earlier: 
$$ \liminf_{x\in K, \|x\| \to \infty} \left[ {\inf_{w \in
			 \Phi(x)} w^T(x-\xref)}\right] > 	0.$$
This contradiction implies that $L_<(K,\Phi)$ is bounded and by Theorem~\ref{Th:L<bdd}, SVI$(K,\Phi)$ is solvable.
\end{proof}



\subsection{SQVIs with single-valued mappings} \label{subsect:sqvi-single}
Our first result is an extension of \cite[Cor.~2.8.4]{facchinei02finite}
to the stochastic regime. In particular, we assume that the mapping
$\bkE\left[G(x;\om)\right]$ cannot be directly obtained; instead, we
provide an existence statement that relies on the scenario-based map
$G(x;\om)$.  

\begin{proposition}[Solvability of SQVI($K,F$)]\label{exist-qvi}\em
 Suppose Assumptions  \ref{assumption1} and
 \ref{assumption-det-stratset} hold. Furthermore, suppose there exists 
a bounded open set $U \subset \Real^n$ and a vector $\xref \in U$  such
that the following hold:
\begin{enumerate}
\item[(a)] For every $\xbar \in \textrm{cl}(U)$, the image $K(\xbar)$ is nonempty and  $\displaystyle \lim_{x \to \xbar} K(x)=K(\xbar);$
\item[(b)] $\xref \in K(x)$ for every $ x \in \textrm{cl}(U);$
\item[(c)] $L_<(K,G;{\om}) \cap \textrm{bd}(U)  = \emptyset $ holds almost surely, where
\begin{align}\label{eqn:setLG} L_<(K,G;{\om}) \triangleq  \displaystyle
\left\{ x \in K(x)~|~(x-\xref)^T G(x;\om) < 0  \displaystyle
\right\}.\end{align}
\end{enumerate}
Then, SQVI($K,F$) has a solution {where $F(x)=\bkE\left[ G(x;\om) \right].$}
\end{proposition}
\begin{proof}
{Recall that by~\cite[Cor.~2.8.4]{facchinei02finite}, the stochastic SQVI($K,F$) is
solvable if }
$ L_{<}(K,F) \cap \textrm{bd}(U)  = \emptyset,$ where 
$$ L_<(K,F) \triangleq \displaystyle \left\{ x \in K(x)~|~( x-\xref)^T \bkE\left[ G(x;\om) \right]< 0 \displaystyle \right\}.$$
\uvs{We proceed by contradiction and assume that there exists an $x \in
L_<(K,F)$
and $x \in \textrm{bd}(U)$. By assumption, since $L_<(K,G;\omega) \cap
\mbox{bd}(U) = \emptyset$ and  $x \in \mbox{bd}(U)$,
	it follows that $x \notin L_<(K,G;{\om})$ almost surely. This
	implies that   $(x-\xref)^T G(x;\om) \geq 0$ for all $\om \in
\Omega$. {It follows that} $(x-\xref)^T \bkE [G(x;\om)] \geq 0,$
implying that $ x \notin L_<(K,F)$. This contradicts our assertion that $x \in
L_<(K,F)$.  Therefore, we must have that $ L_{<}(K,F) \cap \textrm{bd}(U)  = \emptyset$ and  by \cite[Cor.~2.8.4]{facchinei02finite}, the stochastic
SQVI($K,F$) has a solution.}
\end{proof}
Another avenue for ascertaining existence of equilibrium in stochastic regimes is an extension of Harker's result \cite[Th.~2]{harker91generalized} which we present next.

\begin{theorem}[Solvability of SQVI($K,F$) under
compactness]\label{exist-cpt} \em
Suppose Assumptions  \ref{assumption1}  and
\ref{assumption-det-stratset}  hold and there exists a nonempty compact
convex set $\Gamma$ such that the following hold:
\begin{itemize}
\item[(i)] $K(x) \subseteq \Gamma$, $\forall x \in \Gamma$;
\item[(ii)] $K$ is a nonempty, continuous, closed and convex-valued mapping on $\Gamma$. 
\end{itemize}
Then the SQVI($K,F$) has a solution.
\end{theorem}
\begin{proof}
Since $F$ is continuous by Assumption \ref{assumption1}(ii), 
all conditions of~\cite[Th.~2]{harker91generalized} hold. Thus, the SQVI($K,F$) has a solution. 
\end{proof}


The above theorem relies on properties of the map $K$ and the continuity of the map $F$ to ascertain existence of solution. By Assumption \ref{assumption1}, continuity of the map $F$ holds in the settings we consider and thus the solvability of SQVI($K,F$) follows readily. This theorem has a slightly different flavor compared to other results in this paper in the sense that we do not look at properties of the scenario-based map (other than continuity)  that then guarantee existence of solution. We have listed this theorem here for completeness as it presents an alternate perspective of looking at the question of solvability of  SQVI($K,F$).


\subsection{SQVIs with multi-valued mappings}\label{subsect:sqvi-multi}
In this section, we relax the assumption of single-valuedness of the
scenario-based mappings \uma{$G(\bullet;\om)$} and instead allow for the map
$\Phi(\bullet;\om)$ to be multi-valued. In the spirit of the rest of
this paper, our interest lies in deriving results that do not rely on
evaluation of expectation. We use the concepts of set-valued integrals
discussed in the previous section \ref{ss:set-valued background} and
require that Assumption \ref{assumption-det-stratset} holds throughout
this subsection.  Our first existence result relies on a sufficiency
condition for generalized quasi-variational
inequalities~\cite[Cor.~3.1]{1982ChanPang_GQVI}. We recall
~\cite[Cor.~3.1]{1982ChanPang_GQVI}} which can be applied to the
multi-valued SQVI($K,\Phi$).

\begin{proposition}[{\cite[Cor.~3.1]{1982ChanPang_GQVI}}]\label{exist-ns-det} \em
Consider the SQVI($K,\Phi$). 
Suppose that there exists a nonempty compact convex set $C$ such that
{the following hold:}
\begin{enumerate}
\item[(a)] $K(C) \subseteq C$;
\item[(b)] $\bkE[ \Phi(x;\om)]$ is a nonempty contractible{-valued}
{and} compact-valued upper semicontinuous mapping on $C$;
\item[(c)] $K$ is nonempty continuous convex-valued mapping on $C$.
\end{enumerate}
Then the stochastic SQVI($K,\Phi$) admits a solution
\label{prop-suff-cond-Cor3.1}
\end{proposition}
However, this result requires evaluating $\bkE[\Phi(x;\om)]$, an object that
admits far less tractability; instead, we develop almost-sure
sufficiency conditions that imply the requirements of Proposition~\ref{prop-suff-cond-Cor3.1}.
\begin{proposition}[Solvability of SQVI($K,\Phi)$)]\em
Suppose Assumptions \ref{lsc} ans \ref{assumption-det-stratset}  hold.
Furthermore, suppose there exists a nonempty compact convex set $C$ such that
{the following hold:}
\begin{enumerate}
\item[(a)] $K(C) \subseteq C$;
\item[(b)] $\Phi(x;\om)$ is a nonempty  upper semicontinuous mapping for $x \in C$ in
an almost-sure sense;
\item[(c)] $K$ is nonempty, continuous and convex-valued mapping on $C$.
\end{enumerate}
Then the stochastic SQVI($K,\Phi$) admits a solution.
\label{prop:suff-cond-stochastic}
\end{proposition}
\begin{proof}
From Proposition~\ref{exist-ns-det}, it suffices to show that under the
above assumptions, {$\bkE [\Phi(x;\om)]$ is a nonempty
	contractible-valued, compact-valued, upper semicontinuous mapping on
		$C$. }
\begin{enumerate}
\item[(i)] {$\bkE [\Phi(x;\om)]$ is nonempty:}
{Since $\Phi(x;\om)$ is lower semicontinuous, it is a measurable set-valued map. Since it is a measurable set-valued map with
		nonempty  closed images,  by Aumann's measurable selection
		theorem \cite[Th.~8.1.3]{aubin90setvalued}, there exists a
		measurable selection $h$ from $\Phi(x;\om)$. Since $\Phi(x;\om)$
		is integrably bounded, every measurable selection is integrable.
		Thus, $\int_\Omega h d\bkP \in \int_\Omega \Phi(x;\om) $,
		implying that $\bkE [\Phi(x;\om)]$ is nonempty.}  
{\item[(ii)] {$\bkE [\Phi(x;\om)]$ is contractible-valued:} Since the
	probability space is nonatomic by definition, we have that $\bkE
		[\Phi(x;\om)]$ is a convex set.  Since a convex set is
		contractible, we have that $\bkE [\Phi(x;\om)]$ is contractible.}
{\item[(iii)] {$\bkE [\Phi(x;\om)]$ is compact-valued:} Since
	$\Phi(x;\om)$ is integrably bounded, by
		\cite[Th.~8.6.3]{aubin90setvalued}, we get that $\bkE
		[\Phi(x;\om)]$ is compact. }
\item[(iv)]{$\bkE [\Phi(x;\om)]$ is upper semicontinuous:}
{By hypothesis, we have that  $\Phi(x;\om)$ is a measurable, integrably
	bounded and upper-semicontinuous $x \in C$. Thus, from
		~\cite[Cor.~5.2]{aumann}, it follows that $\bkE [\Phi(x;\om)]$
		is upper semicontinuous.}
\end{enumerate}
\end{proof}

{The previous result relies on the compact-valuedness of $K$ with
	respect to a compact set $C$, a property that cannot be universally
		guaranteed. An alternate result for deterministic generalized QVI
		problems \cite[Cor.~4.1]{1982ChanPang_GQVI} leverages coercivity
		properties of the map $\Phi(x)$ to claim existence of a
		solution. We state this result next. 
\begin{proposition}[{\cite[Cor.~4.1]{1982ChanPang_GQVI}}]\label{prop-ns-coerc-det}\em Let $K$ be a set-valued
map from $\Real^n$ to subsets of $\Real^n$ and $\Phi$ from $\Real^n$ to
subsets of $\Real^n$ be a measurable  integrably bounded set-valued map
with closed nonempty images. Suppose that there exists a vector $\xref$
such that $$\xref \in \underset{x \in \textrm{dom}(K)}{\bigcap} K(x)$$ 
\begin{align}
\mbox{ and }
\lim_{\|x\| \to \infty, x \in K(x)} \left[ \inf_{y \in \Phi(x)} {(x
		- \xref)^T y \over{\|x\|} }\right]= \infty.
\label{eqn:coercivity-det}
\end{align}
Suppose the following hold:
\begin{itemize}
\item[(i)] $ \Phi(x)$ is a nonempty, contractible-valued,
	compact-valued, upper semicontinuous map on $\Real^n$;
\item[(ii)] $K$ is convex-valued;
\item[(iii)] There exists a $\rho_0 > 0$ such that $K(x) \cap B_{\rho}$
is a continuous mapping for all $\rho  > \rho_0 $ where $B_{\rho}$ is a
ball of radius $\rho$ centered at the origin. 
\end{itemize}
Then for each vector $q$, SQVI($K,\Phi+q$)  has a solution. Moreover, there exists an $r >0$ such that $\|x^*\| < r$ for each solution $(x^*,y^*)$.
\end{proposition}

In the next proposition, by using the properties of $\Phi(x;\om)$, we
develop an integration-free analog of this result for multi-valued
SQVI$(K,\Phi)$.

\begin{proposition}[Solvability of
SQVI($K,\Phi$)]\label{prop-exist-ns}\em
Suppose Assumptions \ref{lsc} and \ref{assumption-det-stratset} hold. Suppose that there exists a vector $\xref$ such that 
\begin{itemize}
\item[(i)] $\xref \in \underset{x \in \textrm{dom}(K)}{\bigcap} K(x);$
\item[(ii)]
\begin{align}
\lim_{\|x\| \to \infty, x \in K(x)}\left[\inf_{ y \in  \Phi(x;\om)} {(x
		- \xref)^T y \over{\|x\|} }\right]= \infty, \,\mbox{ a.s.}
\label{eqn:coercivity}
\end{align}
\item[(iii)]  For the above $\xref$,  suppose  there exists a nonnegative integrable function $U(\om)$ such that  $ g(x;\om)^T(x-\xref) \geq -U(\om)$ holds  almost surely for any integrable selection $g(x;\om)$ of $\Phi(x;\om)$ and for any $x$.
\item[(iv)] $\Phi(x;\om)$ is an upper semicontinuous mapping on
$\Real^n$ in an almost-sure sense; 
\item[(v)] There exists a $\rho_0 > 0$ such that $K(x) \cap B_{\rho}$ is
a continuous mapping for all $\rho  > \rho_0.$
\end{itemize}
Then for each vector $q$ the stochastic SQVI($K,\Phi+q$)  has a solution. Moreover, there exists an $r >0$ such that $\|x^*\| < r$ for each solution $(x^*,y^*)$.
\end{proposition}
\begin{proof}
First we show that {\eqref{eqn:coercivity}} implies that
the coercivity property \eqref{eqn:coercivity-det} holds for the
expected-value map $\Phi(x)=\bkE[\Phi(x;\om)]$. {We proceed by
	contradiction and assume that} $\bkE[\Phi(x;\om)]$ does not
satisfy \eqref{eqn:coercivity-det}. Then, there exists a
{subsequence $\{x_k\}_{k \in \Kscr}$} such that 
{\begin{align*}
\lim_{\substack{\|x_k\| \to \infty,\\
	x_k \in K(x_k), k \in \Kscr}}\left[\inf_{y \in \bkE [\Phi(x_k;\om)]}
{(x_k- \xref)^T y \over{\|x_k\|} } \right]< \infty. 
\end{align*}
In other words, we have that 
\begin{align*}
\liminf_{\substack{\|x_k\| \to \infty,\\
	x_k \in K(x_k)}}\left[\inf_{y \in \bkE [\Phi(x_k;\om)]}
{(x_k- \xref)^T y \over{\|x_k\|} } \right]< \infty. 
\end{align*}}
Since $\Phi(x_k;\om)$ is integrably bounded, we have that $\bkE
[\Phi(x_k;\om)]$ is compact (by \cite[Th.~8.6.3]{aubin90setvalued}) and therefore it is a closed set. Thus,  we
may conclude that there exists a $y_k \in\bkE [\Phi(x_k;\om)]$ for which the infimum
is attained and the above statement can be rewritten as follows:
{\begin{align*}
\liminf_{\substack{\|x_k\| \to \infty,\\
	x_k \in K(x_k)}} \left[ {(x_k- \xref)^T y_k \over{\|x_k\|} } \right] < \infty, \mbox{ where } y_k \in \bkE [\Phi(x_k;\om)].
\end{align*}}
By Lemma ~\ref{lemma:selections}, $y_k = \int_\Omega g_k(x_k;\om) d\bkP$   where $g_k(x_k;\om)= \sum_{j=0}^n \lambda_j(x_k) f_j(x_k;\om)$ where $\lambda_j(x_k) \geq 0, \sum_{j=0}^n\lambda_j(x_k)=1$ and each  $f_j(x_k;\om)$ is an extremal selection of  $\Phi(x_k;\om)$. Thus, we can write the above limit as
{\begin{align*}
\liminf_{\substack{\|x_k\| \to \infty,\\
	x_k \in K(x_k)}} \left[\int_{\Omega} {(x_k- \xref)^T g_k(x_k;\om)  \over{\|x_k\|} }d\bkP \right]< \infty. 
\end{align*}}
Since, $f_j(x_k;\om)$ is integrable for every $j = 0, \hdots, n$, each
$g_k(x_k;\om)$ is integrable.  Hypothesis (iii) allows for the
application of  Fatou's lemma, through which we have that the following
sequence of inequalities:
{\begin{align*}
& \int_{\Omega} \liminf_{\substack{\|x_k\| \to \infty,\\
	x_k \in K(x_k)}} \left[{(x_k- \xref)^T g_k(x_k;\om)\over{\|x_k\|} } \right]d\bkP 
 \leq \liminf_{\substack{\|x_k\| \to \infty,\\
	x_k \in K(x_k)}}  \left[ \int_{\Omega}{(x_k-
		\xref)^T g_k(x_k;\om)\over{\|x_k\|} }d\bkP \right]
 < \infty.
\end{align*}}
But this implies that the integrand must be finite almost surely. In
other words,  we obtain that 
{ \begin{align*}
\liminf_{\substack{\|x_k\| \to \infty,\\
	x_k \in K(x_k)}}\left[{(x_k- \xref)^T
	g_k(x_k;\om)\over{\|x_k\|} } \right]< \infty, \mbox{ a.s.}
\end{align*}
As a consequence,  
\begin{align*}
	\infty 	&  > \quad  \left[\liminf_{\substack{\|x_k\| \to \infty,\\
	x_k \in K(x_k)}} \left[{\uvs{g_k}(x_k;\omega)^T(x_k- \xref)}
 \over{\|x_k\|} 
	\right]\right] 
	  \geq \left[\liminf_{\substack{\|x_k\| \to \infty,\\
	x_k \in K(x_k)}} \left[\inf_{y_k \in
		\Phi(x_k;\omega)}{(x_k- \xref)^Ty_k
 \over{\|x_k\|} }
	\right]\right]\quad  \mbox{ a.s.}. 
\end{align*}
But this contradicts hypothesis ~\eqref{eqn:coercivity}.
It follows that the coercivity requirement
	~\eqref{eqn:coercivity-det} holds for $\Phi(x) =\bkE[\Phi(x;\om)]$
	and the required existence result follows.}

\gap

Further, from the proof of the
Proposition~\ref{prop:suff-cond-stochastic}, we may claim that
$\bkE[\Phi(x;\om)]$ is a nonempty contractible-valued, compact-valued,
	upper semicontinuous mapping on $\Real^n$ and from Assumption
	\ref{assumption-det-stratset}, the map $K$ is convex-valued. Thus,
	all the hypotheses of Proposition~\ref{prop-ns-coerc-det} are
	satisfied and the multi-valued SQVI($K,\Phi$) admits a solution. 
\end{proof}

\section{Stochastic complementarity problems}\label{sect:SCP}
When the set $K$ in a VI($K,F$) is a cone in
$\Real^n$, then the VI($K,F$) is equivalent to a
CP($K,F$)~\cite{Karmardian1971GCP}.  Our approach in the previous
sections required us to assume that the map $K$ was deterministic. In
practical settings, however, the map $K$ may take on a variety of forms. For
instance, $K$ may be defined by a set of algebraic resource
or budget constraints in financial applications, capacity constraints in
network settings or supply and demand constraints in economic
equilibrium settings. Naturally, these constraints may often be
expectation or risk-based constraints. In such an
instance, a complementarity approach assumes relevance. {Consider
an optimization problem with expectation constraints:
\begin{align}
\begin{aligned}
	\min & \quad \mathbb{E}[f(x;\omega)] \\
	\st & \quad \mathbb{E}[c(x,\omega)] \leq 0, \qquad (\lambda) \\
		&   \qquad \qquad \  x \geq 0. 
\end{aligned}
\end{align}
where $f(x,\omega)$ and $c(x,\omega)$ are convex and continuously
differentiable functions in $x$ for every $\omega$. Then, under a
suitable regularity condition and allowing for the interchange of
derivative and expectations, the necessary and sufficient conditions
of optimality are given by 
\begin{align} \left[
\begin{aligned}
0 \ \leq \ x  & \ \perp \ \mathbb{E}[ \nabla_x f(x;\omega) + \nabla_x
	c(x;\omega)^T \lambda] \ \geq \ 0 \\
0 \ \leq \ \lambda & \ \perp \ -\mathbb{E}[c(x;\omega)]  \ \geq \ 0. 
\end{aligned} \right] \equiv \left[ \Real_{m+n}^+ \ \ni \ z \ \perp \
F(z) \ \in \ \Real_{m+n}^+ \right],  
\end{align}
where $z \triangleq (x;\lambda)$ and $F(z)$ is defined as follows:
$$ F(z) \triangleq \pmat{\mathbb{E}[ \nabla_x f(x;\omega) + \nabla_x
	c(x;\omega)^T \lambda] \\-\mathbb{E}[c(x;\omega)]}.$$}
Specifically, in
such a case, this problem is defined in the joint space of
primal variables and the Lagrange multipliers corresponding to the
stochastic constraints.  Such a transformation yields a stochastic
complementarity problem 
SCP$(K,F)$ where the map $F$ may be expectation-valued while the set $K$
is a deterministic cone. However, such complementarity problems may also
arise naturally, as is the case when modeling frictional contact
problems ~\cite{facchinei02finite} and stochastic counterparts of such
problems emerge from attempting to model risk and uncertainty. In the
remainder of this section, we consider complementarity problems with
single-valued maps. Before proceeding, we provide a set of definitions.
\begin{definition}[CP$(K,q,M)$]\em
Given a cone $K$ in $\Real^n$, an $n \times n$ matrix $M$ and a vector $q \in \Real^n$, the complementarity problem CP$(K,q,M)$ requires an $x \in K, Mx + q \in K^*$ such that $x^T(Mx+q) = 0$.
\end{definition}
Recall, from section~\ref{sec:21}, $K^* \triangleq \{y: y^Td \geq 0, \,
	\forall d \in K\}$. The recession cone, denoted by $K_{\infty}$, is
	defined as follows. 
\begin{definition}[Recession cone $K_{\infty}$]\em
The recession cone associated with a set $K$ (not necessairly a cone) is defined as 
$ K_{\infty} \triangleq \{d: \mbox{for some } x \in K, \{x + \tau d: \tau \geq 0\} \subseteq K\}.$
\end{definition}
Note that when $K$ is a closed and convex cone, $K_{\infty} = K$. Next,
	 we define the CP kernel of a pair $(K,M)$ and define its ${\bf
		 R}_0$ variant. 
\begin{definition}[CP kernel of the pair ($K,M$) ($\Kscr(K,M)$)]\em
The CP kernel of the pair ($K,M$) denoted by $\Kscr(K,M)$ is given by $\Kscr(K,M) = \mbox{SOL}(K_{\infty},0,M).$
\end{definition}

\begin{definition}[{\bf R}$_0$ pair ($K,M$)]\em
($K,M$) is said to be an {\bf R}$_0$ pair if $\Kscr(K,M) = \{0\}.$
\end{definition}

From ~\cite[Th.~2.5.6]{facchinei02finite}, when $K$ is a closed and convex cone, ($K,M$) is an {\bf R}$_0$ pair if and only if the solutions of the CP($K,q,M$) are uniformly bounded for all $q$ belonging to a bounded set.  
\begin{definition}\em
Let $K$ be a cone in $\Real^n$ and $M$ be an $n \times n$ matrix. Then $M$ is said to be
\begin{itemize}
\item[(a)] { {copositive on $K$}} if $x^TMx ~\geq 0, \quad \forall ~x \in K$;
\item[(b)]{ {strictly copositive on $K$}} if $x^TMx ~> 0, \quad \forall ~x \in K \backslash \{0\}$.
\end{itemize}
\end{definition}

\gap

The main result in this section is an almost-sure sufficiency condition
for the solvability of a stochastic complementarity problem SCP$(K,F)$. 
We now state the following sufficiency
condition~\cite[Th.~2.6.1]{facchinei02finite} for the solvability of
deterministic complementarity problems which is subsequently used in
analyzing the stochastic generalizations.
\begin{theorem} \label{Th2.6.1}\em
Let $K$ be a nonempty, closed, and convex cone in $\Real^n$ and let $F$ be a continuous
map from $K$ into $\Real^n$. If there exists a copositive matrix $E \in
\Real^{n\times n}$ on $K$ such that ($K, E$) is an {\bf R}$_0$  pair and the union
$$\bigcup_{\tau >0} \textrm{SOL}(K, F + \tau E)$$
is bounded, then the CP($K, F$) has a solution.
\end{theorem}
Recall that in our notation for SCP($K,F$), $F(x)=\bkE[G(x;\om)]$. 

We are now prepared to prove our main result.

\begin{theorem}[Solvability of SCP($K,F$)]\label{prop:SCPSolvability} \em
 Consider the stochastic complementarity problem SCP($K, F$). Suppose Assumption \ref{assumption1} holds
 for the mapping $F$. Further suppose either (i.a) or (i.b)
	hold in addition to (ii):
 \begin{itemize}
\item[(i.a)] Suppose $K \triangleq \Real_n^+$ and $M$ is a {nonzero}  copositive
			 matrix on $K$ such that  $(\Real_n^+,{M})$ is an {\bf
				 R}$_0$ pair. In addition, suppose the following holds:
\begin{align}\label{eqn:H_lowerbound1}
 {\liminf_{x\in K, \|x\| \to \infty } G(x;\omega) > 0, \mbox{ a.s.}}
\end{align}
\item[(i.b)] Suppose $K$ is a nonempty convex cone and $M$ is a {nonzero}  copositive
			 matrix on $K$ such that  $(K,{M})$ is an {\bf
				 R}$_0$ pair. In addition, suppose the following holds:
\begin{align}\label{eqn:SCPhypo}
\liminf_{x \in K, \|x\| \to \infty} \left[x^TG(x;\om) \right] >  0,~
\mbox{ a.s.}
\end{align}
\item[(ii)]  Suppose there exists a nonnegative integrable function
{$u(\omega)$} such that $x^TG(x;\om)  \geq -u(\om)$ holds almost surely
for any $x$. 
\end{itemize}
Then the stochastic complementarity problem SCP($K, F$) admits a solution.
\end{theorem}
\begin{proof}
\noindent {\bf (i.a.) and (ii) hold.} Note that if
\eqref{eqn:H_lowerbound1}  holds,  since  {$K$ is the
		nonnegative orthant and }$\|x_k\| \to \infty$ as $k \to \infty$,
			{we must have} $\|x_k\|> 0$ for sufficiently large $k$. 
{From  hypothesis \eqref{eqn:H_lowerbound1}, we have that
	\eqref{eqn:xH_lowerbound} holds:}
	\begin{align}
	 {\liminf_{\|x\| \to \infty, x\in K }
		 \left[ x^TG(x;\omega) \right] > 0,\mbox{ a.s.} }
		 \label{eqn:xH_lowerbound}
		 \end{align}
We proceed to show that the set $T$ is bounded where $(\Real_n^+,M)$ is
an ${\bf R}_0$ pair and $T$ is defined as 
\begin{align}\label{TomegaBounded}
   T \triangleq \bigcup_{\tau >0}\textrm{SOL}(K,{F}+\tau M). 
\end{align}
It suffices to show that there exists an $m > 0$ such that for all $x
\in K, \|x\|> m$ implies $x \notin T$.  Suppose there is no such finite
$m$, implying that 
\begin{align}\label{eqnSCP} \mbox{ for any } m > 0 , \ \exists \ x \ \in \ K, \|x \| \ > \  m \mbox{ and } x \ \in \ T.  \end{align}
For each $k > 0$, choose $x_k \in K$  such that $\|x_k\| > k$ and
 $ x_k
\in T$.  For this sequence $\|x_k\| \to \infty$.  Since $ \|x_k\| > k$
observe that for any $k$, $x_k \neq 0$.  Now, for each $k$, since $x_k
\in T$, it follows that $x_k \in \textrm{SOL}(K,F+\tau_k M)$ for some
$\tau_k > 0 $.  Thus, for each $k$ we   have $x_k^TF(x_k) + \tau_k
x_k^TMx_k = 0$. Since $x_k \neq 0$ this means that for each $k$, 
	\begin{align}\label{eqn:Negative}
	x_k^TF(x_k) =  \bkE[x_k^TG(x_k;\om)] = - \tau_k x_k^TMx_k . 
	\end{align}
	Observe that,  since $x_k \in K$ and $\|x_k \| \to \infty$,  we have
	$x_k \neq 0$. Further, since $M$ is copositive we have $ x_k^TMx_k
	\geq
	0$. Since $\|x_k \| \to \infty$, we have that $ x_k^TMx_k\geq 0$.
	Since
	$\tau_k > 0$, there are {three} possibilities for the sequence
	$\tau_k
	x_k^TMx_k$: {$\tau_k
		x_k^TMx_k \to +\infty$, $\tau_k
			x_k^TMx_k \to \epsilon > 0$, or $\tau_k
			x_k^TMx_k \to 0$ as $k \to \infty$}. In {any of these
				cases}, as $k
				\to \infty$ from \eqref{eqn:Negative} we can conclude 
				\begin{align}\label{eqn:Neg}
				\liminf_{k \to \infty} ~ x_k^TF(x_k) =  \liminf_{k \to
					\infty} ~ \left[\bkE[x_k^TG(x_k;\om)]\right]
					=\liminf_{k \to \infty} \left[ - \tau_k x_k^TMx_k
					\right]  \leq 0.
					\end{align}
					On the other hand, by \eqref{eqn:xH_lowerbound} we
					have that 
					\begin{align}\label{eqn:lbd}
					\liminf_{k \to \infty, {x_k \in K}}
					~x_k^TG(x_k;\om) > 0~ \mbox{ a.s.}
					\end{align}
					By hypothesis (ii), Fatou's lemma can be applied,
					giving us 
					\begin{align}\label{eqn:Positive}
					\liminf_{k \to \infty} ~ x_k^TF(x_k) = \liminf_{k
						\to \infty}
						~\left[\bkE\left[x_k^TG(x_k;\om)\right]
						\right]\geq \bkE\left[ \liminf_{k \to \infty}
						x_k^TG(x_k;\om)\right] > 0,
						\end{align}
						where the last inequality follows from
						\eqref{eqn:lbd}. But this
						contradicts \eqref{eqn:Neg} and implies that
						there is a scalar $m$ such
						that $x \in K, \|x\|> m$ implies $x \notin T$.
						In other words, $T$ is
						bounded. We have shown that all the conditions
						of  Theorem~\ref{Th2.6.1}
						are satisfied and we may conclude that the
						stochastic complementarity problem SCP($K, F$)
	has a solution.

\gap

\noindent {{\bf (i.b.) and (ii) hold.} 
We proceed to show that the set $T$ is bounded where $(K,M)$ is an ${\bf
	R}_0$ pair and $T$ is defined as follows:
\begin{align}\label{TomegaBounded}
 T \triangleq \bigcup_{\tau >0}\textrm{SOL}(K,F+\tau M). 
\end{align}
As earlier, it suffices to show that there exists an $m > 0$ such that for all $x \in K, \|x\|> m$ implies $x \notin T$.
Suppose there is no such $m$ implying  that
\begin{align}\label{eqnSCP}
\mbox{ for any } m > 0 ~\exists x \in K, \|x \| > m \mbox{ and }  x \in T.
\end{align}
Construct a sequence $\{x_k\}$ as follows: 
For each $m=k > 0$, choose $x_k \in K$  such that $x_k \in K , \|x_k\| > k$ and $ x_k \in T$. For this sequence $\|x_k\| \to \infty$.
Since $ \|x_k\| > k$ observe that for any $k$, $x_k \neq 0$.
Now, for each $k$ since  $x_k \in T$,   it follows that $x_k \in \textrm{SOL}(K,F+\tau_k M)$ for some $\tau_k > 0 $. Thus, for each $k$, we   have $x_k> 0$, $F(x_k)+\tau_k Mx_k \geq 0, \tau_k > 0$ and $x_k^TF(x_k) + \tau_k x_k^TMx_k = 0$. Since $x_k \neq 0$ and $M$ is copositive and $(K,M)$ is an ${\bf R}_0$ pair we have that $x_k^T Mx_k \geq 0 $. This together with $\tau_k > 0$ means that for each $k$, 
\begin{align*}
\bkE[x_k^TG(x_k;\om)] = x_k^TF(x_k) = - \tau_k x_k^TMx_k  \leq 0.
\end{align*}
This gives us that
\begin{align}\label{eqn:negExp}
\liminf_{k \to \infty} ~ \bkE[x_k^TG(x_k;\om)] = \liminf_{k \to \infty} ~ \left[- \tau_k x_k^TMx_k \right] \leq 0.
\end{align}
On the other hand,  since $x_k \in K$ and $\|x_k \| \to \infty$, by hypothesis \eqref{eqn:SCPhypo} we have that $$\liminf_{k \to \infty} ~x_k^TG(x_k;\om) > 0~ \mbox{ a.s.}$$
This means that 
\begin{align}\label{eqn:exp_pos}
\bkE \left[ \liminf_{k \to \infty} ~x_k^TG(x_k;\om) \right]  > 0.
\end{align}
Now, by hypothesis (ii), Fatou's lemma is applicable, implying that
$$\liminf_{k \to \infty} \bkE \left[ x_k^TG(x_k;\om) \right] \geq \bkE \left[ \liminf_{k \to \infty} ~x_k^TH(x_k;\om) \right]  > 0,$$
where the last inequality follows from \eqref{eqn:exp_pos}. This contradicts \eqref{eqn:negExp}. This contradiction implies that there exists an $m$ such that $x \in K, \|x\|> m$ implies $x \notin T$. In other words, $T$ is bounded.
By hypothesis, we have that there exists a copositive matrix $M$ on $K$
such that  ($K,M$) is an ${\bf R}_0$ pair and we have shown that $T$ is
bounded. Thus, all conditions of  Theorem~\ref{Th2.6.1} are satisfied
and we may conclude that the stochastic complementarity problem SCP($K,
		F$) has a solution. }
\end{proof}
\noindent {\bf Remark.} \uvs{It is worth emphasizing that while hypothesis
(i.a) implies hypothesis (i.b) holds, in general, if $(\Real_n^+,M)$ is
an ${\bf R}_0$ pair, it is not true that the same matrix $M$ forms an
$(K,M)$ is an ${\bf R}_0$ pair.}

\gap

We now consider several corollaries, the first of which requires defining
a \textit{co-coercive mapping.}
\begin{definition}[Co-coercive function]\em
A mapping $F: K \subseteq \Real^n \to \Real^n$ is said to be co-coercive on $K$ if there exists a constant $c > 0 $ such that
$$ (F(x)-F(y))^T (x-y) \geq c {\|F(x)-F(y)\|}^2,~ \forall x,y \in K.$$
\end{definition}

\gap

We state Cor.~\cite[Cor. 2.6.3]{facchinei02finite}, which is used in the proof of the next proposition. 

\begin{corollary}\label{cor:2.6.3}\em
Let $K$ be a pointed, closed, convex cone in $\Real^n$ and let $\us{G}:
\Real^n \to \Real^n$ be a continuous map. If $F$ is co-coercive on
$\Real^n$, then CP$(K,F)$ has a nonempty compact solution set if and
only if there exists a vector $u \in \Real^n$ satisfying $F(u) \in
\textrm{ int}(K^*)$.
\end{corollary}

Our next result provides sufficiency conditions for the existence of a
solution to an SCP when an additional co-coercivity assumption is imposed on
the mapping. In particular, we assume that the mapping $H(z;\om)$ is
co-coercive for almost every $\omega \in \Omega.$ 
\begin{proposition}[Solvability under co-coercivity]\em
Let $K$ be a pointed, closed and convex cone in $\Real^n$. Suppose
Assumption~\ref{assumption1} holds for the mapping $G(x;\om)$ and $G(x;\omega)$ is co-coercive on $K$ {with constant
		$\eta(\omega) > 0$.  Suppose $\eta(\omega) \geq \bar\eta >0$ for
	all $\omega \in \bar \Omega$  where $\bkP(\bar \Omega) = 1$} and  there
	exists a deterministic vector $u \in \Real^{n}$ satisfying $G(u;\om) \in
	\textrm{int}(K^*)$  in an almost sure sense. Then the solution set of the SCP($K, F$) is a nonempty
	and compact set.
\end{proposition}
\begin{proof}
First we show that, $\bkE \left [G(x;\om)\right ]$ is co-coercive in
	$x$.  We begin by noting that  
{\begin{align*}
 (x-y)^T \left( \bkE\left[ G(x;\om)\right ] -\bkE\left [G(y;\om)\right ] \right) 
 & = \int_{\Omega}( x-y)^T (G(x;\om) - G(y;\om) ) d\bkP\\
 & =  \int_{\bar \Omega}( x-y)^T (G(x;\om) - G(y;\om) ) d\bkP,
\end{align*}
 where the second equality follows by noting that $\mathbb{P}(\bar
		 \Omega) = 1$. This allows us to conclude that
  \begin{align*}\int_{\bar \Omega}( x-y)^T (G(x;\om) - G(y;\om) ) d\bkP
 & \geq   \int_{\bar \Omega}  \eta(\omega) \|G(x;\om) - G(y;\om) \|^2 d\bkP\\
 & \geq  \bar \eta \int_{\bar \Omega}  \left  \|G(x;\om) - G(y;\om)
 \right \|^2 d\bkP,
\end{align*}}\noindent where the first  inequality follows from the co-coercivity of
$G(x;\omega)$, the second inequality follows from noting that 
$\eta(\omega) \geq \bar \eta$ for $\omega \in \bar \Omega$, a set of
unitary measure. Finally, by again recalling that $\bar \Omega$ has
measure one and by leveraging Jensen's inequality since ${\|.\|}^2$ is a convex
function, the required result follows:
$$ \bar \eta \int_{\bar \Omega}  \left  \|G(x;\om) - G(y;\om)  \right
\|^2 d\bkP = \bar \eta \int_{\Omega}  \left  \|G(x;\om) - G(y;\om)  \right \|^2 d\bkP
 \geq \bar \eta \left \|\bkE  \left [G(x;\om)\right ] - \bkE \left
 [G(y;\om)\right ]\right \|^2.$$

Further, since $G(u;\om) \in \textrm{int}(K^*)$ holds almost surely for
a deterministic vector $u$, we have that for all $x \in K$, $G(u;\om)^T
x \geq 0 $ holds almost surely. This implies that for all $x \in K, \bkE
\left[G(u;\om)\right]^T x \geq 0 $ holds. Thus, there exists a $u \in
\Real^{n}$ such that $ \bkE \left[G(u;\om)\right] \in K^*$. 

\gap

{It remains 
to show that }$\bkE\left[G(u;\om)\right]$ lies in $\textrm{int} (K^*)$.
If $\bkE\left[ G(u;\om)\right] \notin \textrm{int}(K^*)$, then there exists an $x
\in K$ such that $ \bkE\left [ G(u;\om)\right]^T x = 0 $. Since $x \in
K$ and by assumption, $G(u;\om) \in  \textrm{int} (K^*)$ almost surely,
	it follows that $G(u;\om)^T x > 0 $ almost surely, implying that $ \bkE \left[H(u;\om)\right]^T x > 0$. Thus,
we arrive at a contradiction, proving  that $\bkE \left[G(u;\om) \right] \in
\textrm{int}(K^*)$. Thus, by Corollary ~\ref{cor:2.6.3}, since $ \bkE\left[
H(z;\om)\right]$ is co-coercive and there is a vector $u \in \Real^{n}$ such that
$\bkE\left[ G(u; \om)\right] \in \mbox{int} (K^*),$ it follows that SCP($K,F$) has a
nonempty compact solution set. 
\end{proof}
The next corollary is a direct application of Theorem~\ref{Th2.6.1} to
SCP($K,F$) when $E$ is the identity matrix and can be viewed as a
{\em theorem of the alternative} for CPs.
\begin{corollary}[Cor.~2.6.2~\cite{facchinei02finite}]\label{cor:SCPsolIOR}\em
Let $K$ be a closed convex cone in $\Real^n$ and let $G(x;\om)$ satisfy Assumption \ref{assumption1}. 
 Either SCP($K,F$) has a solution or there exists an unbounded sequence of vectors $\{x_k\}$ and a sequence of positive scalars $\{\tau_k\}$ such that for every $k$, the following complementarity condition holds: 
$$ K \ni x_k~ \perp~ \bkE\left[G(x_k;\om)\right]+ \tau_k x_k\in K^*.$$
\end{corollary}
We may leverage this result in deriving a stochastic generalization.
\begin{proposition}[Theorem of the alternative]\em
Let $K$ be a closed convex cone in $\Real^n$ and let $G(x;\om)$ be a
mapping  that satisfies Assumption \ref{assumption1}. Either SCP$(K,F)$
has a solution or there exists an unbounded sequence of vectors $\{x_k\}$ and a sequence of positive scalars $\{\tau_k\}$ such that for every $k$, the following complementarity condition holds almost-surely:
\begin{equation} K \ni x_k~ \perp~ G(x_k;\om)+ \tau_k x_k\in K^*. \label{eq-as}
\end{equation}
\end{proposition}
\begin{proof}
Suppose \eqref{eq-as} holds almost surely. Consequently, it also holds
in expectation or
\begin{equation} K \ni x_k~ \perp~ \mathbb{E}[ G(x_k;\om)]+ \tau_k
x_k\in K^*. \label{eq-exp}
\end{equation}
Therefore by Cor.~\ref{cor:SCPsolIOR}, SCP$(K,F)$ does not admit a
solution.  \end{proof}

\section{Examples revisited}\label{sect:examples_applied}
We now revisit the motivating examples presented in Section
\ref{ss:notation} and show the applicability of the developed
sufficiency conditions in the context of such problems.

\subsection{Stochastic Nash-Cournot games with nonsmooth price functions}\label{ss:ex_ap:STEP}
In Section~\ref{sec:23}, we described a stochastic 
Nash-Cournot game in which the price functions were nonsmooth. We
revisit this example in showing the associated stochastic
quasi-variational inequality problem is solvable.  Before proceeding, we recall that $f_i(x;\omega)$ is a convex function
of $x_i$, given $x_{-i}$ (see~\cite[Lemma~1]{hobbs07nash}).
\begin{lemma}\em
	Consider the function $f_i(x;\om) = c_i(x_i)-x_ip(X;\omega)$ where
	$p(X;\om)$  is
	given by \eqref{eqn:nc_price}. Then $f_i(x_i;\xni)$ is a convex function in $x_i$ for all
	$\xni$. 
\end{lemma}
The convexity  of $f_i$ and $K_i(x_{-i})$ allows us to claim that the
first-order optimality conditions are sufficient; these conditions are
given by a multi-valued quasi-variational inequality  SQVI($K,\Phi$)
	where $\Phi(x)$, the Cartesian product of  generalized gradients, is defined as
	$$\Phi(x) \triangleq\bkE \left[\prod_{i=1}^N \partial_{x_i}
	f_i(x;\om)\right],$$
	and  $\Phi(x;\om)$ is defined as $\prod_{i=1}^N \partial_{x_i}
	f_i(x;\om).$  The subdifferential set of $f_i(x;\omega)$ is defined
	as 
\begin{align*}
    \partial_{x_i}f_{i}(x;\om) 
    &= c_i'(x_i) - \partial_{x_i}(x_{i}p(X;\om))
     = c_i'(x_i) - p(X;\om) - x_i \partial_{x_i} p(X;\om). 
\end{align*}
Thus, if  $w \in \Phi(x;\om)$, then $w
	= \prod_{i=1}^n w_i$ where $w_i \in \partial_{x_i}f_{i}(x;\om) $.
Based on the piecewise smooth nature of $p(X;\om)$,
the Clarke generalized gradient of $p$ is defined as follows:
\begin{align}
	\partial_{x_i} p(X;\om) \in  \begin{cases}
								 	\{-b_1(\omega)\}, & 0 \leq  X < \beta^1 \\
									-[b^{j-1}(\omega),b^{j}(\omega)],
									& \beta^{j-1} = X, \, j =
									2, \hdots, s \\
									\{-b^s(\omega)\},& \beta^s
									< X
								\end{cases}
	\end{align}
Since our interest lies in showing the applicability of our sufficiency
conditions when the map $\Phi$ is expectation valued, we impose the
required assumptions on the map ${\bf K}$ as captured by
Prop.~\ref{prop-exist-ns} (i) and (v). Existence of a nonsmooth
stochastic Nash-Cournot equilibrium follows from showing that hypotheses
(ii) -- (iv) of Prop.~\ref{prop-exist-ns} do indeed hold.

\begin{theorem}[Existence of stochastic Nash-Cournot equilibrium]\em
Consider the stochastic generalized Nash-Cournot game  and suppose
Assumptions \ref{lsc}, \ref{assumption-det-stratset} and  \ref{nc} hold.
Further, assume \us{that conditions (i) and (v) of } Prop.~\ref{prop-exist-ns}  hold. Then,
	this game admits an equilibrium. 
\end{theorem}
\begin{proof}
Since $\partial_{x_i} f_i(x;\om)$ is a Clarke generalized gradient, it
is a nonempty upper semicontinuous mapping at $x_i$, given $x_{-i}$.
Furthermore, the integrability of $(a^j(\omega),b^j(\omega))$ for $j =
1, \hdots, s$ allows us to claim that $\partial_{x_i} f_i(x;\om)$ is
integrably bounded. Consequently, hypothesis (iv) in Proposition  ~\ref{prop-exist-ns} holds.

\gap

By Assumption ~\ref{nc}, since  $a_i(\om)$ and $b_i(\om)$ are positive,
   we have that they are bounded below by the nonnegative constant (integrable) function $0$. From, this and the description of
   $\Phi$ derived above, we see that  hypothesis (iii) in Prop.
   ~\ref{prop-exist-ns} holds. Thus Fatou's lemma can be applied. 

\gap

We now proceed to show that hypothesis (ii) in
Proposition~\ref{prop-exist-ns} holds.
It suffices to show that there exists an $\xref \in \K(x)$ such that 
$$ \lim_{\|x\| \to \infty, x \in \K(x)} \left(\inf_{ w \in \Phi(x;\om)}
\frac{ (x-\xref)^T w}{ \|x\|} \right)= \infty. $$
Consider a vector $\xref$ such that $$\xref \in \bigcap_{x ~\in~ \textrm{dom}(\K)} \K(x).$$ 
Then $w^T(x-\xref)$ can be expressed as the sum of several terms:
\begin{align*}
w^T(x-\xref) & = \sum_{i=1}^N c_i'(x_i)(x_i-\xref_i)  - \, p(X;\om) (x-\xref) - \sum_{i=1}^N x_i(x_i - \xref_i) \partial_{x_i}p(X;\om). 
\end{align*}
When $\|x\| \to \infty$, from the nonnegativity of $x$, it follows that $X \to \infty$ and for
sufficiently large $X$, we have that 
$\partial_{x_i} p(X;\omega) =-b^s(\omega).$ 
Consequently, for almost every $\omega \in \Omega$, we have that 
\begin{align*}
&  \quad \lim_{\|x\| \to \infty, x \in \K(x)} \inf_{w \in \Phi(x;\om)}
\frac{ (x-\xref)^Tw}{\|x\|} \\
& = \lim_{\|x\| \to \infty, x \in \K(x)}\left(\frac{\sum_{i=1}^N (c_i'(x_i) +
			b^s(\omega) x_i)(x_i-\xref_i)}{\|x\|}\right) 
- \lim_{\|x\| \to \infty, x \in \K(x)}\left(\frac{
			   p(X;\om) (x-\xref)}{\|x\|}\right) \\
& = \lim_{\|x\| \to \infty, x \in \K(x)}\overbrace{\left(\frac{\sum_{i=1}^N (c_i'(x_i) +
			b^s(\omega) (X+x_i))(x_i-\xref_i) }{\|x\|}\right)}^{Term \,
	(a)} 
 - \lim_{\|x\| \to \infty, x \in
	\K(x)}\overbrace{\left(\frac{a^s(\omega)
			(x-\xref)}{\|x\|}\right)}^{\it Term\, (b)} = \infty, \end{align*}
where the last equality is a consequence of noting that the numerator of
Term (a) tends to $+\infty$ at a quadratic rate while the numerator of
Term (b) tends to $+\infty$ at a linear rate. The existence of an
equilibrium follows from the application of Prop.~\ref{prop-exist-ns}.
\end{proof}

\subsection{Strategic behavior in power markets}\label{ss:ex_ap:NC}
In Section~\ref{sec:23}, we have presented a model for strategic
behavior in imperfectly competitive electricity markets. We will now
develop a stochastic complementarity-based formulation of such a
problem. The developed sufficiency conditions will then be applied to
this problem.

\gap

Recall that, the resulting problem faced by firm $f$  can be stated as
follows:
\[ \begin{array}{ll} 
\displaystyle{
{\operatornamewithlimits{\mbox{maximize}}_{s_{fi}, \, g_{fi}}}
} \quad \mathbb{E}\left[\displaystyle{
\sum_{i\in \Nscr}
} \, \left( \, p_i(S_i;\om)s_{fi} - c_{fi}(g_{fi};\om)
		-(s_{fi}-g_{fi})w_i \, \right)\right] &\\ [0.2in] 
\begin{array}{lrrllll}
\mbox{subject to} &
    g_{fi} &\leq &\,  {\rm cap}_{fi} & \quad (\, \mu_{fi} \,) , &\quad \forall \, i \in \Nscr \\ [5pt]
&  0  &\leq&\, g_{fi}  , &  & \quad \forall \, i \in \Nscr \\ [5pt]
& 0 &\leq  &\, s_{fi}   ,  & & \quad \forall \,i \in \Nscr  \\  [5pt]
\mbox{and} & \displaystyle{\sum_{i \in \Nscr}} \, ( \, s_{fi} - g_{fi} \, ) \, &= &\, 0. & \quad (\, \lambda_{f} \,) & \end{array}
\end{array}
\]
The equilibrium conditions of this problem are given by the following
complementarity problem.
\[
\begin{array}{lllr}
0 \leq \, s_{fi} & \perp & \mathbb{E} \left[ -p'_{i}(S_i;\om)s_{fi} -p_i(S_i;\om) + w_i \right] - \lambda_f  \geq 0,  &\quad \forall \,i \in \Nscr  \\[5pt]
0 \leq \, g_{fi} & \perp & \mathbb{E}\left[ c'_{fi}(g_{fi};\om) -w_i \right] + \mu_{fi} + \lambda_{f}  \geq 0,  &\quad \forall \, i \in \Nscr \\ [5pt]
 0 \leq \, \mu_{fi} & \perp &  {\rm cap}_{fi} -g_{fi} \geq 0, & \quad \forall \, i \in \Nscr \\ [5pt]
 \qquad \lambda_f & \perp &  \displaystyle{\sum_{i \in \Nscr}} \, ( \,
		 s_{fi} - g_{fi} \, ) = 0. & \quad \forall \, f \in \Fscr 
\end{array}
\]
The ISO's optimization problem is given by 
\[ \begin{array}{ll} 
\displaystyle{
{\operatornamewithlimits{\mbox{maximize}}_{y_i}}
} & \displaystyle{
\sum_{i \in \Nscr}
} \, y_i w_i \\ [0.2in]
\mbox{subject to} & \displaystyle{
\sum_{i \in \Nscr}
} \, {\rm PDF}_{ij} y_i \, \leq \, T_j, \quad ( \eta_j )\qquad \forall \, j \in \Kscr
\end{array} \]
and its optimality conditions are as follows:
\begin{align}\label{eqn:optcond}
\begin{array}{llr}
w_i \, = &  \displaystyle{ \sum_{j \in \Kscr}}\eta_j {\rm PDF}_{ij} &\qquad \forall \,i \in \Nscr,\\
0 \leq \eta_j  \perp & T_j - \displaystyle{ \sum_{i \in \Nscr}}{\rm PDF}_{ij} y_i  \geq 0 & \qquad \forall \, j \in \Kscr.
\end{array} 
\end{align}
The market clearing conditions are given by the following. 
\[
y_i \, = \, \displaystyle{
\sum_{h \in \Fscr}
} \, \left( \, s_{hi} - g_{hi} \, \right), \quad \qquad \qquad \quad \quad\forall \, i \in \Nscr.
\] 
Next, we define $\ell_i(\omega)$ and $h_i(\omega)$ as follows:
\begin{align} \label{eqn:kappa}
\ell_i(\om) & = -p'_{i}(S_i;\om)s_{fi} -p_i(S_i;\om) + w_i =
-p'_{i}(S_i;\om)s_{fi} -p_i(S_i;\om)  + \displaystyle{ \sum_{j \in
	\Kscr}}\eta_j {\rm PDF}_{ij} \\
\label{eqn:h_i}
h_i(\om) & = c'_{fi}(g_{fi};\om) -w_i = c'_{fi}(g_{fi};\om) - \displaystyle{ \sum_{j \in \Kscr}}\eta_j {\rm PDF}_{ij}.
\end{align} 
 Then, by aggregating all the equilibrium conditions together and eliminating $w_i$ and $y_i$ based on the  equality constraints \eqref{eqn:optcond}, we get the equilibrium conditions in $s_{fi},g_{fi}, \mu_{fi}, , \lambda_f,$ and $\eta_j$ are as follows
\[
\left\{\begin{array}{lllrl}
 0 \leq \, s_{fi} & \perp & \mathbb{E} \left[\ell_i(\om)\right] - \lambda_f  &\geq 0,  &\quad \forall \,i \in \Nscr  \\[5pt]
0 \leq \, g_{fi} & \perp & \mathbb{E}\left[h_i(\om)\right] + \mu_{fi} + \lambda_{f} &\geq 0,  &\quad \forall \, i \in \Nscr \\ [5pt]
0 \leq \, \mu_{fi} & \perp &  {\rm cap}_{fi} -g_{fi} &\geq 0, &\quad \forall \, i \in \Nscr \\ [5pt]
\qquad \lambda_f & \perp & \displaystyle{\sum_{i \in \Nscr}} \, ( \,
		s_{fi} - g_{fi} \, )  & =0, & 
 \end{array} \right\}, \quad \forall \,f\in \Fscr \\ \\
 \]
%
\[
\begin{array}{lllrl}
\mbox{ and } \, 0 \leq \eta_j  & \perp & T_j - \displaystyle{ \sum_{i \in \Nscr}}{\rm PDF}_{ij} \displaystyle{
\sum_{h \in \Fscr}
} \, \left( \, s_{hi} - g_{hi} \, \right)  &\geq 0.& \qquad \forall \, j \in \Kscr
\end{array}
\]
This can be viewed as the following stochastic (mixed)-complementarity problem where
$x$, $B$, and $G(x;\omega)$ are appropriately defined:
\begin{align*}
	0 \, \leq \, x & \, \perp \, \mathbb{E}[G(x;\omega)] - B^T \lambda
	\, \geq \, 0 \\
		\lambda  & \, \perp \, Bx \, =\,  0.
\end{align*}
It follows that the inner product in the coercivity
condition~\eqref{eqn:SCPhypo} reduces to $x^T \mathbb{E}[G(x;\omega)]$
as observed by this simplification:
$$ \pmat{ x\\ \lambda}^T \pmat{ \mathbb{E}[G(x;\omega)] - B^T \lambda \\
	Bx } = x^T \mathbb{E}[G(x;\omega)]. $$
Next, we show that this inner product is bounded from below by
$-u(\omega)$ where  $u(\omega)$ is a nonnegative integrable function. 
\begin{lemma}\label{supp-1}\em
For the stochastic complementarity problem SCP($K,F$) above that represents the strategic behavior in power markets, there exists a nonnegative integrable function $u(\om)$ such that  have that the following holds:
$$ G(x;\om) = x^T G(x;\om) \geq -u(\om) \mbox{ almost surely for all } x \ \in \ K.  $$
\end{lemma}
\begin{proof}
The product $x^T G(x;\om)$ can be expressed as follows:
\begin{align*}
& \quad \displaystyle{ \sum_{f, i}} \left( -p'_{i}(S_i;\om)s^2_{fi} -p_i(S_i;\om)s_{fi} +  \left( \displaystyle{ \sum_{j \in \Kscr}}\eta_j {\rm PDF}_{ij} \right) s_{fi}   \right) \\ 
& + \displaystyle{ \sum_{f, i}} \left( c'_{fi}(g_{fi};\om) g_{fi}- \left( \displaystyle{ \sum_{j \in \Kscr}}\eta_j {\rm PDF}_{ij} \right)g_{fi}  + \mu_{fi} g_{fi}  \right)   \\ 
& + \displaystyle{ \sum_{f, i}} \left( \mu_{fi} {\rm cap}_{fi} -\mu_{fi}g_{fi} \right)  + \displaystyle{ \sum_{j}}\eta_j \left( T_j - \displaystyle{ \sum_{i \in \Nscr}}{\rm PDF}_{ij} \displaystyle{
\sum_{h \in \Fscr}
} \, \left( \, s_{hi} - g_{hi} \, \right)   \right).
\end{align*}
After appropriate cancellations, this reduces to 
\begin{align*}
\displaystyle{ \sum_{f, i}} \left( -p'_{i}(S_i;\om)s^2_{fi}
		-p_i(S_i;\om)s_{fi}  \right) + \displaystyle{ \sum_{f, i}}
\left( c'_{fi}(g_{fi};\om) g_{fi} \right)  + \displaystyle{ \sum_{f, i}}
\left( \mu_{fi} {\rm cap}_{fi} \right)  +  \displaystyle{
	\sum_{j}}\eta_j  T_j.   
\end{align*}
By Assumption \ref{assump:power}, the price
functions are decreasing functions bounded above by an integrable
function and the cost functions are non-decreasing. Furthermore, {$K$
	is the nonnegative orthant, $\mu_{fi}, \eta_j$ are nonnegative,
	   and $ {\rm cap}_{fi} , T_j$ denote nonnegative capacities}.
	   Consequently, we have the following sequence of inequalities.
	\begin{align*}
& \displaystyle{ \sum_{f, i}} \left( -p'_{i}(S_i;\om)s^2_{fi}
		-p_i(S_i;\om)s_{fi}  \right) + \displaystyle{ \sum_{f, i}}
\left( c'_{fi}(g_{fi};\om) g_{fi} \right)  + \displaystyle{ \sum_{f, i}}
\left( \mu_{fi} {\rm cap}_{fi} \right)  +  \displaystyle{
	\sum_{j}}\eta_j  T_j \\
	& \geq \displaystyle{ \sum_{f, i}} \left( 		-p_i(S_i;\om)s_{fi}
			\right)  \geq -\left(\max_{i} \bar p_i(\omega) \right) \sum_{f,i} \mbox{cap}_{fi} \triangleq
	-u(\omega), 
\end{align*}
where $p_i(S_i;\omega) \leq \bar p_i(\omega)$ for all nonnegative $S_i$ and
$\sum_{f,i} s_{fi} \leq \sum_{f,i} \mbox{cap}_{fi}.$  Integrability of $u(\omega)$ follows immediately by its definition. 
\end{proof}

Having presented the supporting results, we now prove the existence of
an equilibrium.

\begin{proposition}[Existence of an imperfectly competitive equilibrium] \em Consider the
imperfectly competitive model in  power markets.  Under  Assumption
\ref{assump:power}, this problem admits a solution.
\end{proposition}
\begin{proof}
The result follows by showing that Theorem~\ref{prop:SCPSolvability} can be
applied.  Lemma~\ref{supp-1} shows that \us{hypothesis (ii) of }
Theprem~\ref{prop:SCPSolvability} holds. We proceed to show that
\us{hypothesis (i.b) of}  Theorem~\ref{prop:SCPSolvability} also holds. We show that the following property holds almost surely:
 \begin{align}
 {\liminf_{\|x\| \to \infty,~ x\geq 0}
x^T G(x;\omega)  > 0. }
\label{eqn:lowerbd}
\end{align}

Consider the expression for $x^TG(x;\om)$ derived in Lemma~\ref{supp-1}. 

\begin{align*}
x^TG(x;\om)= \displaystyle{ \sum_{f, i}} \left( -p'_{i}(S_i;\om)s^2_{fi} -p_i(S_i;\om)s_{fi}  \right) + \displaystyle{ \sum_{f, i}} \left( c'_{fi}(g_{fi};\om) g_{fi} \right)  + \displaystyle{ \sum_{f, i}} \left( \mu_{fi} {\rm cap}_{fi} \right)  +  \displaystyle{ \sum_{j}}\eta_j  T_j.  
\end{align*}

For large $\|x\|$, the first summation is dominated by its first term
and by Assumption \ref{assump:power}, as $\|x\|$ goes to $\infty$, this
term goes to $\infty$. The other terms are all nonnegative by Assumption
\ref{assump:power}. Thus, the entire expression can only increase to
$\infty$ as $\|x\|$ goes to $\infty$. This proves that \eqref{eqn:lowerbd}
holds and the required result follows. 

\end{proof}

\section{Concluding Remarks} \label{sect:conclusion}
Finite-dimensional variational inequality and complementarity problems
have proved to be extraordinarily useful tools for modeling a range of
equilibrium problems in engineering, economics, and finance. This avenue
of study is facilitated by the presence of a comprehensive theory for
the solvability of variational inequality problems and their variants.
When such problems are complicated by uncertainty, a subclass of models
lead to variational problems whose maps contain expectations. A direct
application of available theory requires access to analytical forms of
such integrals and their derivatives, severely limiting the utility of
existing sufficiency conditions for solvability.

\gap

To resolve this gap, we provide a set of integration-free sufficiency
conditions for the existence of solutions to variational inequality
problems, quasi-variational generalizations, and complementarity
problems in settings where the maps are either single-valued or
multi-valued. These conditions find utility in the existence of
equilibria in the context of generalized nonsmooth stochastic
Nash-Cournot games and strategic problems in power markets. We believe
that these statements are but a first step in examining a range of
problems in stochastic regimes. These include the development of
stability and sensitivity statements as well as the consideration of
broader mathematical objects such as stochastic differential variational
inequality problems. 

\section{Appendix}\label{Appendix}
{\bf Proof of Prop.~\ref{prop-smooth-exist}.}
\begin{proof}
 Recall from~\cite[Ch.~2]{facchinei02finite} that  the solvability of
	SVI$(K,F)$ requires the existence of an $\xref$ such that 
\begin{align}\label{eqn:coer}
	 \liminf_{ \|x\|\rightarrow \infty, x \in K} \left[ F(x)^T(x-\xref) \right]> 0.
\end{align}
	But we have that 
\begin{align*}
  \liminf_{ \|x\|\rightarrow \infty, x \in K} \left[F(x)^T(x-\xref)
	\right]  =   \liminf_{ \|x\|\rightarrow \infty, x \in K}
	\left[\int_\Omega G(x;\om)^T(x-\xref) d\bkP\right].
\end{align*}
By hypothesis (ii), we may apply Fatou's lemma to obtain the
	following inequality:
\begin{align*}
\liminf_{ \|x\|\rightarrow \infty, x \in K} \left[ F(x)^T(x-\xref) \right]\geq
	\int_\Omega \liminf_{ \|x\|\rightarrow \infty, x \in
		K}\left[G(x;\om)^T(x-\xref)\right] d\bkP > 0, 
\end{align*}
where the last inequality follows from the given hypothesis. {Thus \eqref{eqn:coer} holds and SVI$(K,F)$ has a
	solution.}\end{proof}

{\bf Proof of Prop.~\ref{prop-smooth-exist-cart}.}
\begin{proof}
For the given $\xref \in K$ and for any $x \in K$, there exists a $\nu
\in \{1, \hdots, N\}$, such that 
$$\liminf_{\|x_\nu\|\to \infty, x_\nu\in K_\nu} \left[ \uvs{G}_\nu(x;\om)^T(x_\nu - \xref_\nu) \right]> 0 $$
holds almost surely. Thus we obtain
$$\bkE \left[\liminf_{\|x_\nu\|\to \infty, x_\nu\in K_\nu}
\uvs{G}_\nu(x;\om)^T(x_\nu - \xref_\nu) \right ] > 0. $$
By hypothesis (ii) above we may apply Fatou's lemma to get
$$ \liminf_{\|x_\nu\|\to \infty, x_\nu\in K_\nu} \bkE
\left[\uvs{G}_\nu(x;\om)^T(x_\nu - \xref_\nu)\right]  > 0. $$
This implies that $C_{\leq}$ is bounded where 
$$ C_{\leq} := \left\{ x \in K : \max_{1 \leq \nu \leq N} \bkE \left[
\uvs{G}_\nu(x;\om)^T(x_\nu - \xref_\nu)\right] \leq 0\right\}.$$ 

From~\cite[Prop.~3.5.1]{facchinei02finite}, boundedness of  $C_{\leq}$
allows us to conclude that SVI$(K,F)$ is solvable.
\end{proof}
{\bf Proof of Cor.~\ref{prop-monotone-exist}.}\begin{proof}
We begin with the observation that the monotonicity of $G(x;\om)$
allows us to bound $G(x;\om)^T(x-\xref)$ from below as follows:  
\begin{align*}
  G(x;\om)^T(x-\xref) &= \left[ G(x;\om)-G(\xref;\om) \right] ^T (x-\xref)
 + G(\xref;\om) ^T (x-\xref)
 \geq  \uvs{G}({\xref};\om )^T(x-\xref).
\end{align*}
By utilizing this observation, we may complete the proof by using a
similar avenue as followed in Prop.~\ref{prop-smooth-exist} (details
		omitted).
\end{proof}
\bibliographystyle{siam}
\bibliography{ref_14}
\end{document}